\documentclass[10pt,a4paper]{amsart}
\usepackage{amsfonts}
\usepackage{amsthm}
\usepackage{amsmath}
\usepackage{amscd}
\usepackage[latin2]{inputenc}
\usepackage{t1enc}
\usepackage[mathscr]{eucal}
\usepackage{indentfirst}
\usepackage{graphicx}
\usepackage{graphics}
\usepackage{pict2e}
\usepackage{epic}
\usepackage[normalem]{ulem}
\numberwithin{equation}{section}
\usepackage{epstopdf}
\usepackage{verbatim}
\usepackage{amssymb}
\usepackage{tikz-cd}

\usepackage{slashed}

\usepackage{hyperref}
\AtBeginDocument{}

\theoremstyle{plain}
\newtheorem{Th}{Theorem}[section]
\newtheorem{Lemma}[Th]{Lemma}

\newtheorem{Prop}[Th]{Proposition}

\theoremstyle{definition}

\newtheorem{Rem}[Th]{Remark}
\newtheorem{?}[Th]{Problem}

\def\al{\alpha}
\def\be{\beta}

\def\w{\wedge}
\def\R{\mathbb{R}}
\def\C{\mathbb{C}}
\def\CP{\mathbb{C}\mathbb{P}}

\def\Lm{\Lambda}
\def\lm{\lambda}
\def\om{\omega}
\def\Om{\Omega}
\def\vp{\varphi}

\def\U{\mathrm{U}}
\def\SO{\mathrm{SO}}
\def\SU{\mathrm{SU}}
\def\Sp{\mathrm{Sp}}
\def\G{\mathrm{G}}
\def\spin{\mathrm{Spin}}
\def\Spin{\mathrm{Spin}}

\def\ip{\raise1pt\hbox{\large$\lrcorner$}\>}

\DeclareMathOperator{\vol}{vol}

\usepackage{geometry}

\newgeometry{ hmargin={25mm,25mm}}   

\begin{document}
	
\title{Examples of deformed Spin(7)-instantons/Donaldson-Thomas connections}

\author[U. Fowdar]{Udhav Fowdar}

\address{
Dipartimento di Matematica G. Peano, Universit\`{a} di Torino \\  
Via Carlo Alberto 10 \\ 
10123 Torino\\ 
Italy
}

\email{
udhav.fowdar@unito.it, udhav21@gmail.com
}

\keywords{Exceptional holonomy, deformed Donaldson-Thomas connections, mirror symmetry} 
\subjclass[2010]{MSC 53C07, 53C26, 53C29}	
	
\begin{abstract} 
We construct examples of deformed Hermitian Yang-Mills connections and deformed $\Spin(7)$-instantons (also called $\spin(7)$ deformed Donaldson-Thomas connections) on the cotangent bundle of $\C\mathbb{P}^2$ endowed with the Calabi hyperK\"ahler structure. Deformed $\spin(7)$-instantons on
cones over $3$-Sasakian $7$-manifolds are also constructed. We
show that these can be used to distinguish between isometric structures and also between $\Sp(2)$ and $\spin(7)$ holonomy cones. 
To the best of our knowledge, these are the first non-trivial examples of deformed $\spin(7)$-instantons.
\end{abstract}

\maketitle
\tableofcontents

\section{Introduction}
In this paper we construct cohomogeneity one examples of deformed $\spin(7)$-instantons (also called $\spin(7)$ deformed Donaldson-Thomas connections \cite{KawaiMirror2022, KawaiYamamoto22}) on $T^*\C\mathbb{P}^2$ endowed with its $\SU(3)$-invariant Calabi hyperK\"ahler structure \cite{CalabiAnsatz1979}. We study their relations with the deformed Hermitian Yang-Mills connections associated to the underlying hyperK\"ahler $2$-forms. Furthermore, we show that the deformed $\spin(7)$ instantons on the hyperK\"ahler cone are different to those on the $\spin(7)$ cone obtained when the $3$-Sasakian link is equipped with the squashed metric. 
We also construct deformed Hermitian Yang-Mills connections on $T^*\C\mathbb{P}^1$ and the flag manifold $\SU(3)/\mathbb{T}^2$ to illustrate certain similarities and differences with our examples on $T^*\C\mathbb{P}^2$. 
Before stating our  results more precisely we first recall a few basic notions and elaborate on the motivation for this work.

Given a K\"ahler manifold $(M^{2n},g,J,\om)$ and a Hermitian line bundle $L\to M$, a connection  $1$-form 
$A$ on $L$ is called a deformed Hermitian Yang-Mills (dHYM) connection with phase $e^{i\theta}$ if its curvature 
$F_A$, viewed as a \textit{real} $2$-form, satisfies
\begin{gather}
	F_A^{0,2}=0\label{equ: general dhym equation holo},\\
\mathrm{Im}((\om+iF_A)^n)=\tan(\theta)\mathrm{Re}((\om+iF_A)^n),\label{equ: general dhym equation}
\end{gather}
for some constant $\theta$. 
The dHYM equations (\ref{equ: general dhym equation holo})-(\ref{equ: general dhym equation}) were introduced 
in \cite{Leung2000, Marino2000} where it was shown that these
arise naturally in the context of the SYZ conjecture for mirror symmetry \cite{SYZ1996}. Roughly, the SYZ conjecture asserts that mirror pairs of Calabi-Yau $3$-folds admit special Lagrangian torus fibrations in certain limits. 
Solutions to the dHYM equations arise as mirrors to special Lagrangian sections of these fibrations via a real Fourier-Mukai transform cf. \cite[Section 1.2]{Collins2018deformed}. 
Thus, the SYZ conjecture suggests an intricate relation between the calibrated geometry and gauge theory of mirror pairs of Calabi-Yau $3$-folds. 
The SYZ conjecture also has higher dimensional analogues in the context of $\mathrm{G}_2$ and $\Spin(7)$ geometry \cite{GYZ}, and a similar argument as in the Calabi-Yau case leads to the notion of deformed $\mathrm{G}_2$- and deformed $\Spin(7)$-instantons as mirrors to (co)-associatives and Cayley submanifolds respectively via the Fourier-Mukai transform 
cf. \cite{Karigiannis2009, KawaiMirror2022, KawaiYamamoto22, LeeLeung2009}. 

In recent years, dHYM connections have received considerable interest, see for instance \cite{CollinsYaustability, Collins2018deformed, AdamJacob2022, JacobYau2017} and the references therein, and there are by now many known examples on compact K\"ahler manifolds; especially provided by the work of Chen in \cite{Chen2021} where their existence is related to a notion of stability. 
By contrast, very little is known in the context of $\G_2$ and $\Spin(7)$ geometry. The first non-trivial examples of deformed $\G_2$-instantons were constructed by Lotay-Oliveira in \cite{Lotay2020} over certain nearly parallel $\G_2$ manifolds (these are special classes of Einstein manifolds whose cone metrics have holonomy contained in $\spin(7)$). 
Here `non-trivial' means that these connections are not flat and do not arise via pullback from lower-dimensional constructions. 
Subsequently in \cite{FowdardG2}, we constructed non-trivial examples of deformed $\G_2$-instantons over the spinor bundle of $S^3$ endowed with the complete Brandhuber et al. $\G_2$ holonomy metric \cite{Brandhuber01} and also the conical Bryant-Salamon metric \cite{Bryant1989}. These currently exhaust all the known non-trivial examples of deformed $\G_2$-instantons. To the best of our knowledge, there are at present no known non-trivial examples of deformed $\spin(7)$-instantons and it is precisely this problem that motivated this work.

Given a $\spin(7)$ manifold $(M^8,\Phi)$ and a Hermitian line bundle $L \to M$, a connection $A$ on $L$ is called a deformed $\spin(7)$-instanton if $F_A$ satisfies:
\begin{gather}
	\pi^4_7(F_A\w F_A) =0,\label{equ: dspin7 I}\\
	\pi^2_7(F_A-\frac{1}{6}* F_A^3)=0,\label{equ: dspin7 II}
\end{gather}
where $\pi^2_7$ and $\pi^4_7$ are projection maps determined by the $\spin(7)$ $4$-form $\Phi$, see Section \ref{section: preliminaries}. If $M$ is a Calabi-Yau $4$-fold then one can view (\ref{equ: dspin7 I})-(\ref{equ: dspin7 II}) as a generalisation of (\ref{equ: general dhym equation holo})-(\ref{equ: general dhym equation}) in the case when 
$\theta=0$ cf. \cite[Theorem 7.5]{KawaiMirror2022} and \cite[Lemma 7.2]{KawaiYamamoto22}. 
In this paper we consider the deformed $\spin(7)$-instanton equations (\ref{equ: dspin7 I})-(\ref{equ: dspin7 II}) on $ T^*\C\mathbb{P}^2$ endowed with its well-known hyperK\"ahler metric arising from the Calabi ansatz \cite{CalabiAnsatz1979}.
This set up is of particular interest for several reasons: (1) Since the Calabi hyperK\"ahler structure is invariant under a cohomogeneity one action of $\SU(3)$, working in this invariant set up we have that the complicated systems of PDEs given by (\ref{equ: general dhym equation holo})-(\ref{equ: general dhym equation}) and (\ref{equ: dspin7 I})-(\ref{equ: dspin7 II}) reduce to more tractable systems of ODEs. 
(2) As $T^*\C\mathbb{P}^2$ is a hyperK\"ahler manifold, it admits a triple of orthogonal K\"ahler forms $\om_1,\om_2,\om_3$. Hence one can consider the dHYM equations with respect to each K\"ahler form and compare how the solutions vary. (3) We can also view $T^*\C\mathbb{P}^2$ as a $\spin(7)$ manifold with respect to different $\spin(7)$ $4$-forms, defined by (\ref{sp2andspin7i}), and again we can compare the deformed $\spin(7)$-instantons  in each case. (4) Closely related to $T^*\C\mathbb{P}^2$ is the spinor bundle of $\C\mathbb{P}^2$ which only exists as an orbifold. The latter
admits a $\spin(7)$ holonomy metric due to the Bryant-Salamon construction in \cite{Bryant1989} and moreover, it is also invariant under the cohomogeneity one action of $\SU(3)$. As such we can study the deformed $\spin(7)$-instanton equations in this set up as well. 
Note that the Bryant-Salamon and Calabi metrics are both asymptotic to a conical metric over a $3$-Sasakian $7$-manifold; however, in the latter case the metric has holonomy $\Sp(2)$ and in the former case $\spin(7)$. Thus, we can compare the deformed $\spin(7)$-instantons over each cone. With this in mind we can now elaborate on our main results.

\textbf{Statement of main results.} 
Let $\om_1,\om_2,\om_3$ denote the hyperK\"ahler triple on $T^*\C\mathbb{P}^2$ as given in Proposition \ref{prop: HK structure of TCP2} and let $\Phi_1,\Phi_2, \Phi_3$ denote the $\spin(7)$ $4$-forms defined by (\ref{sp2andspin7i}). Parametrising the space of complex line bundles on $T^*\C\mathbb{P}^2$ by $k \in \mathbb{Z} \cong H^2(T^*\C\mathbb{P}^2,\mathbb{Z})$, we have that:
\begin{Th}\label{thm statement 1}
	For generic values of $k$ and $\theta$, there are $2$ distinct $\SU(3)$-invariant dHYM connections on $T^*\C\mathbb{P}^2$ with respect to $\om_1$. On the other hand, for generic values of $k$ and $\theta$, there are $4$ distinct $\SU(3)$-invariant dHYM connections on $T^*\C\mathbb{P}^2$ with respect to $\om_2$ and $\om_3$. Moreover, when $\theta=0$, there is a unique hyper-holomorphic connection which occurs as a dHYM connection for $\om_1,\om_2$ and $\om_3$.
\end{Th}
An explicit description of these solutions is given in Theorem \ref{theorem: dhym for om1} and \ref{theorem: dhym for om2}, see also Remark \ref{rem: dhym om1} and \ref{rem: dhym om2}. Thus, the moduli spaces of $\SU(3)$-invariant dHYM connections with respect to the K\"ahler forms $\om_1,\om_2,\om_3$ are all distinct and they intersect in exactly one point, namely, the hyper-holomorphic connection. It is also worth pointing out that the K\"ahler form $\om_1$ is distinguished from $\om_2$ and $\om_3$ by the fact that it is non-trivial in cohomology. 
Since $\om_1,\om_2,\om_3$ are all compatible with the Calabi metric $g_{Ca}$, this shows that the dHYM connections distinguish between isometric K\"ahler structures. A similar result was demonstrated by Lotay-Oliveira in \cite{Lotay2020} whereby they showed that deformed $\G_2$-instantons can be used to distinguish between isometric $\G_2$-structures on $3$-Sasakian $7$-manifolds. In the deformed $\spin(7)$ case for each fixed $k$ we show that:
\begin{Th}\label{thm statement 2}
The moduli spaces of deformed $\spin(7)$-instantons with respect to $\Phi_1,\Phi_2,\Phi_3$ all consist of a $0$-dimensional component comprising of the hyper-holomorphic connection and all the other components are $1$-dimensional. For $\Phi_1$, there are $2$ one dimensional components while for $\Phi_2$ and $\Phi_3$ there are $4$ one dimensional components.
\end{Th}
An explicit description of these solutions is given in Theorem \ref{theorem: dspin7 for phi1} and \ref{theorem: dspin7 for phi2}. As before we see that the moduli spaces of deformed $\spin(7)$-instantons distinguish between isometric $\spin(7)$-structures. It turns out that all the above deformed $\spin(7)$-instantons also occur as dHYM connections with respect to certain K\"ahler forms lying in the $2$-sphere generated by $\langle\om_1,\om_2,\om_3\rangle$ and also for different phase angle $\theta$, see Remark \ref{rem: dspin7 phi1} and \ref{rem: dspin7 phi2}. In particular, this shows that deformed $\spin(7)$-instantons are not equivalent to dHYM connections with phase $1$ when $M$ has holonomy contained in $\SU(4)$ but is non-compact; this is rather different from the compact situation as shown by Kawai-Yamamoto in \cite[Theorem 7.5]{KawaiMirror2022}, see also Theorem \ref{thm: result of kawai-yamamoto} (2) below. An analogous result was shown by Papoulias in \cite{Papoulias2021} where the author showed that $\spin(7)$-instantons and traceless HYM connections do not coincide on the Calabi-Yau manifold $T^*S^4$; this is again in contrast with the compact case as demonstrated by Lewis in \cite[Theorem 3.1]{Lewis1999spin}, see also \cite[Remark 7.6]{KawaiMirror2022}. This discrepancy is essentially due to the fact that the arguments in the compact cases in \cite{KawaiMirror2022} and \cite{Lewis1999spin} both rely on certain energy estimates and these do not hold in the non-compact setting.

Consider now the Aloff-Wallach space $N_{1,1}=\SU(3)/\U(1)$ which admits both a $3$-Sasakian structure and a strictly nearly parallel $\G_2$-structure.  The cone metric on the former structure has holonomy $\Sp(2)\subset\spin(7)$ while the cone metric for latter has holonomy equal to $\spin(7)$. Denoting the respective $\spin(7)$ $4$-forms by $\Phi_{ts}$ and $\Phi_{np}$ we show that:
\begin{Th}\label{thm statement 3}
	On the trivial line bundle over $\R^+ \times N_{1,1}$ there is a $2$-parameter family of $\SU(3)$-invariant deformed $\spin(7)$-instantons with respect to $\Phi_{ts}$ and a $3$-parameter family with respect to $\Phi_{np}$.
\end{Th}
An explicit description of these solutions is given in Theorem \ref{thm: dspin7 cone} and \ref{thm: dspin7 cone HK}. More generally, the above also applies to any $3$-Sasakian $7$-manifold endowed with its squashed nearly parallel $\G_2$-structure. Thus, this shows that the deformed $\spin(7)$-instantons can distinguish between the $\Sp(2)$ and $\spin(7)$ holonomy cones. Curiously, if we restrict the above deformed $\spin(7)$-instantons to a hypersurface $\{r_0\}\times N_{1,1}$, then it turns out that these coincide, up to constant factor, with the deformed $\G_2$-instantons of Lotay-Oliveira in \cite[Corollary 4.5]{Lotay2020}, see also Remark \ref{rem: analogies}.  Note that $N_{1,1}$ can be viewed as an $\SO(3)$-bundle over $\CP^2$ and in fact corresponds to the principal orbit of $T^*\CP^2$ under the action of $\SU(3)$. Thus, 
the underlying spaces in Theorem \ref{thm statement 1}, \ref{thm statement 2} and \ref{thm statement 3} are all diffeomorphic on the complement of the singular orbit $\CP^2$. Since the Calabi hyperK\"ahler structure converges to the conical $3$-Sasakian structure as the $\CP^2$ shrinks to a point, comparing Theorem \ref{thm statement 2} and \ref{thm statement 3} shows that the dimension of the moduli space can jump under such degenerations. 
 \\

\textbf{Organisation of paper.}
In Section \ref{section: preliminaries} we gather some basic facts about $\Sp(2)$- and $\Spin(7)$-structures on $8$-manifolds. We describe how  
certain irreducible $\spin(7)$-modules split as $\Sp(2)$-modules and describe the relation between various special connections in this set up; these will be useful for the computations in Section \ref{section: main} and \ref{section: bscone}. 
In Section \ref{section: tcp1} we consider the dHYM equations on the cotangent bundle of $T^*\C\mathbb{P}^1$ endowed with the Eguchi-Hanson hyperK\"ahler metric. The latter is invariant under the cohomogeneity one action of $\SU(2)$ and hence this set up provides a simple model for our investigation on $T^*\C\mathbb{P}^2$. We describe all the $\SU(2)$-invariant dHYM connections in Proposition \ref{prop: hym eguchi-hanson} and Theorem \ref{thm: main theorem TCP1}. As in Theorem \ref{thm statement 1} we show that the number of dHYM connections depends on the K\"ahler form and also that there is a unique hyper-holomorphic connection arising in the intersection of all the dHYM connections. 
In Section \ref{section: flag manifold} we construct dHYM connections on the flag manifold $\SU(3)/\mathbb{T}^2$ endowed with its standard K\"ahler Einstein structure, see Theorem \ref{theorem: dhym on flag}. This section serves two main purposes. Firstly it illustrates certain key differences between dHYM connections in the compact and non-compact setting, and secondly it will provide the basic set up for the cohomogeneity one $\SU(3)$-invariant connections that we shall investigate in the subsequent sections. 
In Section \ref{section: main} we consider the dHYM and deformed $\spin(7)$-instantons equations on $T^*\C\mathbb{P}^2$ endowed with the Calabi hyperK\"ahler structure. To this end we begin by classifying the invariant connections on the Aloff-Wallach space $N_{1,1}$ which occurs as the principal orbit for the cohomogeneity one action. 
Using the latter together with the results of Section \ref{section: preliminaries}, we compute the dHYM and deformed $\spin(7)$-instantons equations to obtain a coupled system of ODEs. Surprisingly, these turn out to be completely solvable and we are able to find all the solutions explicitly, see Theorem \ref{theorem: dhym for om1}, \ref{theorem: dhym for om2}, \ref{theorem: dspin7 for phi1} and \ref{theorem: dspin7 for phi2}.
In Section \ref{section: bscone} we consider the deformed $\spin(7)$-instantons equations on the orbi-spinor bundle of $\C\mathbb{P}^2$ endowed with the Bryant-Salamon $\spin(7)$-structure.  
The ODEs in this case turn out to be more complicated than in Section \ref{section: main} and as such we are unable to find the general solution. However, for the conical metric and trivial line bundle we are able to construct an explicit $3$-parameter family of solutions, see Theorem \ref{thm: dspin7 cone}. By contrast, we show that for the conical Calabi structure, there is only a $2$-parameter family of \textit{strict} deformed $\spin(7)$-instantons, see Theorem \ref{thm: dspin7 cone HK}.\\

\textbf{Acknowledgement.} This work was supported by the S\~{a}o Paulo Research Foundation (FAPESP) [2021/07249-0]. The author would like to thank the participants of the Workshop BRIDGES: Specials geometries and gauge theories held at the University of Pau in June 2023 for their valuable feedback that helped improve this article.

\section{Preliminaries}\label{section: preliminaries}

We begin by gathering some basic facts about $\Sp(2)$- and $\Spin(7)$-structures, and their relation. These will be useful in the subsequent sections to compute the deformed instanton equations. 

\textbf{$\Sp(2)$-structures.} An $\Sp(2)$-structure on an $8$-manifold $M^8$ is determined by a triple of $2$-forms $\om_1,\om_2,\om_3$ which at each point can be expressed as
\begin{gather}
\om_1 = dx_{12} + dx_{34} + dx_{56} + dx_{78},\label{equ: om1 pointwise}\\
\om_2 = dx_{13} + dx_{42} + dx_{57} + dx_{86},\label{equ: om2 pointwise}\\
\om_3 = dx_{14} + dx_{23} + dx_{58} + dx_{67},\label{equ: om3 pointwise}
\end{gather}
for some local coordinates  $\{x_i\}$. Note here and also subsequently $dx_{i...j}$ will be used as shorthand for $dx_i\w \cdots \w dx_j$. The $\om_1,\om_2,\om_3$  in turn defines a Riemannian metric $g$ and a triple of almost complex structures $J_i$. More concretely, given any vector field $X$ on $M^8$ we can define $J_i X$ by the relation $X \ip (\om_j^2-\om_k^2)/2= J_iX \ip (\om_j \w \om_k)$, where $(i,j,k)$ denotes a cyclic permutation of $(1,2,3)$. The metric can then be defined by
\begin{equation}
	\om_i(\cdot,\cdot)=g(J_i\cdot,\cdot).\label{kahlermetric}
\end{equation} 
In particular, we can also define a volume form by $\vol=\om_i^4/24$ and hence the Hodge star operator $*$.

$(M^8,g,\om_1,\om_2,\om_3)$ is called a hyperK\"ahler manifold if the holonomy group of its Levi-Civita connection $\nabla$ is contained in $\Sp(2)$. The latter can also be characterised by the condition that $\nabla \om_1=\nabla \om_2=\nabla \om_3=0$, or equivalently, $d \om_1=d \om_2=d \om_3=0$ cf. \cite{Joycebook}.

Given an $\Sp(2)$-structure we can decompose the space of $k$-forms on $M^8$, denoted by $\Lm^k$, into irreducible $\Sp(2)$-modules. For our purpose in this article we shall only need to consider $\Lm^2$ and $\Lm^4$. As an $\Sp(2)$-module we have that
\begin{equation}
	\Lm^2 = \langle \om_1,\om_2,\om_3 \rangle \oplus E_1 \oplus E_2 \oplus E_3 \oplus \Lm^2_{10},
\end{equation}
where
\begin{align}
	\Lm^2_{10}\ =\ &\{\al \in \Lm^2 \ | \ J_1(\al)=J_2(\al)=J_3(\al)=\al \}\\
	\ =\ &\{\al \in \Lm^2 \ | \ *(\al\w \om_i^2)=-2\al \ \ \forall i \}
\end{align}
i.e. $\Lm^2_{10}$ consists of $2$-forms of type $(1,1)$ with respect to $J_1,J_2$ and $J_3$, and similarly we define 
\begin{align}
	E_i \ =\ &\{\al \in \Lm^2 \ | \ J_i(\al)=\al,\  J_j(\al)=J_k(\al)=-\al \text{ and } g(\om_i,\al)=0 \} \label{definition of Ei}\\
	\ =\ &\{\al \in \Lm^2 \ | \ *(\al\w \om_i^2)=-2\al,\ *(\al\w \om_j^2)=+2\al,\ *(\al\w \om_k^2)=+2\al \},
\end{align}
where $(i,j,k)$ denotes a cyclic permutation of $(1,2,3)$. It is worth pointing out that $\Lm^2_{10}$ corresponds to the adjoint representation of $\Sp(2)$ while $E_1 \cong E_2 \cong E_2\cong \R^5$ corresponds to the standard representation of $\SO(5)\cong \Spin(5)/\mathbb{Z}_2\cong \Sp(2)/\mathbb{Z}_2$.

Before considering $\Lm^4$ we first observe that one can define Lefschetz type operators by wedging with $\om_i$ to map $\Lm^2$ into $\Lm^4$, and moreover that these maps are all $\Sp(2)$-invariant. Hence this allows us to split $\Lm^4$ into irreducible $\Sp(2)$-modules as follows:
\begin{align}
	\Lm^4 
	= \langle \om_1^2, \om_2^2, \om_3^2, \om_1 \w \om_2, \om_1 \w \om_3 , \om_2 \w \om_3\rangle \oplus
	\bigoplus_{i=1}^3 F_i^+
	\oplus
	\Lm^{4+}_{14} \oplus
	\Lm^{4-}_5 \oplus
	\bigoplus_{i=1}^3 K_i^-
\end{align}
where
\begin{gather}
	F_i^+ = E_j \w \om_k =  E_k \w \om_j ,\label{definition of F_i^+}\\
	\Lm^{4+}_{14} = \{\al \in \Lm^4 \ | \ \al \w \om_i =0 \ \ \forall i \},\\
	\Lm^{4-}_5 = E_1 \w \om_1 = E_2 \w \om_2 = E_3 \w \om_3,\\
	K_i^- = \Lm^2_{10} \w \om_i,
\end{gather}
and again $(i,j,k)$ denotes a cyclic permutation of $(1,2,3)$ in (\ref{definition of F_i^+}). 
To emphasise on the above notation, we describe, for instance, $F^+_1$ more concretely as follows:  $F_1^+$ is defined as the space of $4$-forms obtained by wedging elements of the irreducible module $E_2$ with $\om_3$, which equivalently corresponds to the space of $4$-forms obtained by wedging elements of $E_3$ with $\om_2$. The superscripts  $+ $ and $-$ refer to the fact that these $4$-forms are self-dual (i.e. $*\al=\al$) and anti-self dual (i.e. $*\al=-\al$)  respectively. In particular, we have that $\dim(F_i^+)=\dim(\Lm^{4-}_5)=5$, $\dim(K^-_i)=10$ and $\dim(\Lm^{4+}_{14})=14$. It is also worth pointing out that $\langle \om_1^2, \om_2^2, \om_3^2\rangle $, $\Lm^{4+}_{14}$ and $\Lm^{4-}_5$ consist precisely of those  $4$-forms which are of type $(1,1)$ with respect to  $J_1,J_2$ and $J_3$.\\

\textbf{Connections on hyperK\"ahler manifolds.} Let $E$ be a Hermitian line bundle over a hyperK\"ahler manifold $M^8$ with
a Hermitian connection $A$. Then we say that $A$ is a Hermitian Yang-Mills (HYM) connection with respect to $\om_i$ if its curvature form $F_A:=dA + A \w A$ satisfies
\begin{gather}
	F_A^{0,2}=0, \label{holomorphiccondition}\\
	F_A \w \om_i^3 = \lm \om^4_i, \label{HYMconstant}
\end{gather}
where the $(0,2)$-component of $F_A$ is taken here with respect to the complex structure $J_i$ and $\lm$ is a constant. Note that equation (\ref{holomorphiccondition}) can also be expressed as 
$$F_A \in  \langle \om_i \rangle \oplus E_i \oplus \Lm^2_{10}$$ 
i.e. $J_i(F_A)=F_A$, which is also equivalent to $F_A \w (\om_j+i\om_k)^2 =0$, and equation (\ref{HYMconstant}) corresponds to the condition that the $\om_i$ component of $F_A$ is constant. When $\lm=0$ i.e. $F_A\in  E_i \oplus \Lm^2_{10}$, such a connection $A$ is called \textit{traceless} Hermitian Yang-Mills with respect to $\om_i$. 
If furthermore, we have that 
$$F_A \in \Lm^2_{10}$$ 
i.e. $F_A$ is of type $(1,1)$ with respect to $J_1$, $J_2$ and $J_3$, then we say that $A$ is a hyper-holomorphic connection. For simplicity, we shall often called a HYM or dHYM connection with respect to the K\"ahler form $\om_i$, an $\om_i$-HYM  or $\om_i$-dHYM connection respectively. 
Observe that in the case of $8$-manifolds, the dHYM equation (\ref{equ: general dhym equation}) can be expressed as
\begin{equation}
	4(\om^3 \w F_A- \om \w F_A^3) = \tan(\theta) (\om^4 - 6\om^2 \w F_A^2 + F_A^4).\label{equ: dhym n=4}
\end{equation}

\begin{Rem}\label{rem: dhym convergence}
	It is not hard to see from the latter expression that if ${A_k}$ denotes a sequence of dHYM connections such that $\|F_{A_k}\|\to 0$ then  $A_k$ convergences to a HYM connection (with $\lm$ as in (\ref{HYMconstant}) determined by $\theta$); this is equivalent to the large volume limit described in \cite[Section 1]{Collins2018deformed} and \cite[Section 5.2]{Marino2000}. An analogous result holds for deformed $\G_2$-instantons cf. \cite[Proposition 4.4]{FowdardG2} and below we shall see that this holds in the $\spin(7)$ case as well.
\end{Rem}

\textbf{$\Spin(7)$-structures.} A $\Spin(7)$-structure on an $8$-manifold $M^8$ is determined by a $4$-form $\Phi$  which at each point can be expressed as
\begin{align}
	\Phi =\ &dx_{1234}+dx_{1256}+dx_{1278}+dx_{1357}+dx_{1386}-dx_{1458}-dx_{1467}\ +\label{equ: spin7 pointwise} \\
	&dx_{5678}+dx_{3478}+dx_{3456}+dx_{4286}+dx_{4257}-dx_{2367}-dx_{2358},\nonumber
\end{align}
for some local coordinates  $\{x_i\}$. Since $\Spin(7)\subset \SO(8)$ it follows that $\Phi$ defines both a Riemannian metric and orientation; in particular, we can again define the Hodge star operator $*$.  $(M^8,\Phi)$ is called a $\Spin(7)$-manifold if the holonomy of the associated Levi-Civita connection $\nabla$ is contained in $\Spin(7)$. The latter can also be characterised by the condition that $\nabla \Phi=0$, or equivalently $d\Phi=0$ cf. \cite{Bryant1987, Salamon1989}.

Similarly as in the $\Sp(2)$ case, we can decompose $\Lm^2$ into irreducible $\Spin(7)$ modules as:
\begin{equation}
	\Lm^2 = \Lm^2_7 \oplus \Lm^2_{21},
\end{equation}
where
\begin{align}
	\Lm^2_7&=\{\al\in \Lm^2\ |\ *(\al \w \Phi)=3\al\},\\
	\Lm^2_{21}&=\{\al\in \Lm^2\ |\ *(\al \w \Phi)=-\al\}.
\end{align}
Also in analogy with the $\Sp(2)$ case, we have that $\Lm^2_{21}$ corresponds to the adjoint representation of $\Spin(7)$ while $\Lm^2_7\cong \R^7$ corresponds to the standard representation of $\SO(7)\cong \Spin(7)/\mathbb{Z}_2$.

On  the other hand, we have that $\Lm^4$ splits as
\begin{align}
	\Lm^4 = \langle \Phi \rangle \oplus \Lm^4_7 \oplus \Lm^4_{27} \oplus \Lm^4_{35},\label{lambda4}
\end{align}
where 
\begin{equation}
	\Lm^4_{35} =\{\al \in \Lm^4\ |\ * \al=-\alpha \}
\end{equation}
and $\Lm^4_7$ is defined as the span of the action of $\Lm^2_7\hookrightarrow \Lm^2 \cong\mathfrak{so}(8) \hookrightarrow \mathfrak{gl}(8)$, viewed as endomorphisms, on $\Phi$. More concretely, denoting the endomorphism action by $\diamond$,  we have that  for a simple $2$-form $\al \w \be$ its action on $\Phi$ is given by
$(\al \w \be) \diamond \Phi := \al \w (\be^{\sharp} \ip \Phi) -  \be \w (\al^{\sharp} \ip \Phi)$.
Finally $\Lm^4_{27}$ is simply defined as the orthogonal complement of  $\langle \Phi\rangle \oplus \Lm^4_7\oplus \Lm^4_{35}$ in $\Lm^4$ cf. \cite{Bryant1987}.

We shall denote by $\pi^k_l$ the orthogonal projection of a $k$-form into the space $\Lm^k_l$. Of particular interest to us are the projection maps $\pi^2_7$ and $\pi^4_7$ since these occur in the definition of the deformed $\Spin(7)$-instanton equations (\ref{equ: dspin7 I}) and (\ref{equ: dspin7 II}). In view of the above characterisation of $\Lm^2_7$ and $\Lm^2_{21}$, it follows that the projection map $\pi^2_7$ can be expressed as
\begin{equation}
	\pi^2_7(\al) =\frac{1}{4} (\al +*(\al \w \Phi)).\label{pi27}
\end{equation}
In contrast we do not have such a nice description for the projection map $\pi^4_{7}$ since the definition of the space $\Lm^4_7$ depends on $\Lm^2_7$ and its action on $\Phi$ by the operator $\diamond$. Below we shall give a more practical method of computing the map $\pi^4_7$ in our framework.\\

\textbf{Connections on $\spin(7)$ manifolds.} Let $E$ be a vector bundle over a $\spin(7)$ manifold $M^8$ with a connection $A$. Then we say that $A$ is a $\Spin(7)$-instanton if 
\begin{equation}
	\pi^2_7(F_A) =0 \Longleftrightarrow *(F_A \w \Phi)=-F_A
\end{equation}
i.e. $F_A \in \Lm^2_{21}$. It is worth pointing out that  as $\Spin(7)$ modules we have that 
$$S^2(\Lm^2_{21}) \cong \langle \Phi \rangle \oplus \Lm^4_{27} \oplus \Lm^4_{35}\oplus V_{0,2,0},$$ 
where $V_{0,2,0}$ denotes an irreducible $168$-dimensional space cf. \cite{Bryant1987}. Since latter decomposition contains no $7$-dimensional module, it follows that $\pi^4_{7}(F_A \w F_A)=0$ for any $\Spin(7)$-instanton $A$; observe that the latter equation is precisely (\ref{equ: dspin7 I}). As in the dHYM case (see Remark \ref{rem: dhym convergence}), it is now easy to see from (\ref{equ: dspin7 II}) that if $A_k$ denotes a sequence of deformed $\spin(7)$-instanton such that $\|F_{A_k}\|\to 0$ then we have that $A_k$ converges to a $\spin(7)$-instanton. 
\begin{Rem}\label{remark: harveylawson}
	Recall that the deformed $\spin(7)$-instanton equations are given by the pair (\ref{equ: dspin7 I})-(\ref{equ: dspin7 II}) and that these two equations are in general independent of each other. 
	However, Kawai-Yamamoto showed in \cite[Proposition 3.3]{Kawai2021deformationdSpin7} that if $A$ satisfies (\ref{equ: dspin7 II}) and additionally we have that $*F_A^4\neq 24$ i.e. $*F_A^4- 24$ is a nowhere vanishing function, then (\ref{equ: dspin7 I}) automatically holds. The latter result can be interpreted as the mirror version of Theorem 2.20 in \cite[Chapter IV]{HarveyLawson} for Cayley submanifolds. More precisely, Cayley graphs in $\R^8$ are defined by mirror equations to (\ref{equ: dspin7 I})-(\ref{equ: dspin7 II}), these correspond to (1)-(2) in \cite[Theorem 2.20]{HarveyLawson}, and Harvey-Lawson showed therein that (1) implies (2) provided the mirror equation to `$*F_A^4\neq 24$' holds. 
	Harvey-Lawson also showed in \cite[Pg. 132]{HarveyLawson} that there are in fact explicit examples when (the mirror condition to) `$*F_A^4=24$' occurs and hence this shows that indeed (1) does not imply (2) in general. In Section \ref{section: main} we shall show that there exist explicit examples of the deformed $\spin(7)$-instantons satisfying $*F_A^4=24$ and as such we cannot appeal to \cite[Proposition 3.3]{Kawai2021deformationdSpin7}. In particular, our examples illustrate that the mirror picture is again consistent. 
\end{Rem}
\textbf{Inclusion of $\Sp(2)$ in $\Spin(7)$.} Given an $\Sp(2)$-structure it is easy to see by comparing the pointwise models given in (\ref{equ: om1 pointwise})-(\ref{equ: om3 pointwise}) and (\ref{equ: spin7 pointwise}) that we can define a $\Spin(7)$ $4$-form by
\begin{align}
	\Phi\ 
	&=\frac{1}{2}(\om_1^2+\om_2^2-\om_3^2).\label{sp2andspin73}
\end{align}
This exhibits $\Sp(2)$ as a subgroup of $\Spin(7)$. Note however that $\Sp(2)$ can be embedded in $\Spin(7)$ in many ways; these embeddings are parametrised by sections of a bundle whose fibre is isomorphic to the homogeneous space $\Spin(7)/\Sp(2)$. Given a hyperK\"ahler triple $\langle\om_1, \om_2, \om_3\rangle$, in this paper we shall be interested in the `canonical' triple of $\Spin(7)$ $4$-forms defined by
\begin{equation}
	\Phi_i := \frac{1}{2}(-\om_i^2+\om_j^2+\om_k^2),\label{sp2andspin7i}
\end{equation}
where $(i,j,k)$ denotes as usual a cyclic permutation of $(1,2,3)$. Observe that $\Phi_3$ corresponds precisely to (\ref{sp2andspin73}). Note also that while each $\Phi_i$ induces the same metric and orientation on $M^8$ (which of course coincide with hyperK\"ahler one), they are all nonetheless distinct as $\Spin(7)$-structures and hence, in particular, will have different $\spin(7)$- and deformed $\spin(7)$-instantons. 

If the $\Spin(7)$ $4$-form is given by (\ref{sp2andspin7i}) then we can express the irreducible $\Spin(7)$ summands of $\Lm^2$ and $\Lm^4$ in terms of the irreducible $\Sp(2)$-modules. A straightforward computation and application of Schur's lemma gives the next result.
\begin{Prop}\label{prop: spin7modulesasSp2}
	The irreducible $\Spin(7)$ summands of $\Lm^2$ and $\Lm^4$ with respect to $\Phi_3$, split into $\Sp(2)$-modules as follows:
\begin{gather*}
	\Lm^2_7 = \langle \om_1, \om_2 \rangle \oplus E_3\\
	\Lm^2_{21} = \langle \om_3 \rangle \oplus E_1 \oplus E_2 \oplus \Lm^2_{10} \\
	\Lm^4_7 = \langle \om_1 \w \om_3, \om_2 \w \om_3 \rangle \oplus F^+_3 \\
	\Lm^4_{27} =  \langle \om_1^2+3\ \!\om_3^2, \om_2^2+3\ \!\om_3^2 , \om_1 \w \om_2 \rangle \oplus F_1^+ \oplus F_2^+ \oplus  \Lm^{4+}_{14}\\
	\Lm^4_{35} =  \Lm^{4-}_5 \oplus K_1^- \oplus K_2^- \oplus K_3^- 
\end{gather*}
If the $\Spin(7)$ $4$-form is instead given by $\Phi_1$ or $\Phi_2$, then it suffices to cyclically permute the indices $1,2,3$ of the $\Sp(2)$-modules in the above decompositions.
\end{Prop}
From the definitions of $E_i$ and the irreducible $\Sp(2)$ summands of $\Lm^4$, it follows for instance that $F_3^+$ can also be characterised as 
\begin{equation}
	F_3^+= \{\al \in \Lm^4 \ |\ *\al =\al, \text{  } \ J_3(\al)=\al,\  J_1(\al)=J_2(\al)=-\al \text{ and } g(\al,\om_1\w \om_2)=0\}.
\end{equation}
This provides a rather practical algorithm to compute the $F_3^+$-component of an arbitrary $4$-form $\al$:
first we self-dualise it by $\al \mapsto \pi^{4+}(\al):=(\al + * \al)/2 $, then we map it to $\pi^{4+}(\al) - J_1(\pi^{4+}(\al))-J_2(\pi^{4+}(\al))+J_3(\pi^{4+}(\al))$ to obtain a self-dual $4$-form invariant by $J_3$ but on which $J_1$ and $J_2$ act as $-\mathrm{Id}$, and finally we project the latter in the orthogonal complement of $\om_1 \w \om_2$. Together with the components of $\al$ in $\om_1 \w \om_3$ and $\om_2 \w \om_3$, this allows us to compute the $\Spin(7)$ projection map $\pi^4_7:\Lm^4 \to \Lm^4_7$.\\ 

\textbf{Inclusion of connections.} 
We now consider the special connections on $(M^8,\om_1,\om_2,\om_3)$ as connections on a manifold with a $\spin(7)$-structure given by (\ref{sp2andspin7i}). 
From Proposition \ref{prop: spin7modulesasSp2} it is easy to see that $A$ is a traceless $\om_i$-HYM connection if and only if it is a $\spin(7)$-instanton both with respect to $\Phi_j$ and $\Phi_k$. Moreover, $A$ is a hyper-holomorphic connection if and only if it is a $\spin(7)$-instanton with respect to the triple $\Phi_1,\Phi_2$ and $\Phi_3$. 
We state the analogous results demonstrated by Kawai-Yamamoto in \cite[Lemma 7.2]{KawaiYamamoto22} and \cite[Theorem 7.5]{KawaiMirror2022} which relate dHYM connections and deformed $\Spin(7)$-instantons in our set up:
\begin{Th}\label{thm: result of kawai-yamamoto}
	Let $(M^8,\om_1,\om_2,\om_3)$ denote a hyperK\"ahler manifold, $A$ a connection on a line bundle $L$, 
	 and $\Phi_i$ the $\Spin(7)$ $4$-form defined by  (\ref{sp2andspin7i}). Then the following hold:
	\begin{enumerate}
	\item If $A$ is a Hermitian connection such that the $(0,2)$-part of $F_A$ with respect to $J_i$ vanishes, then $A$ is an $\om_i$-dHYM connection with phase $1$ if and only if $A$ is a deformed $\Spin(7)$-instanton with respect to $\Phi_k$ (or equivalently $\Phi_j$). 
	\item If $M^8$ is a compact connected manifold, and there exists an $\om_i$-dHYM connection with phase $1$, then any deformed $\spin(7)$-instanton with respect to $\Phi_k$ (or equivalently $\Phi_j$) is a $\om_i$-dHYM connection with phase $1$.
	\end{enumerate}	
\end{Th}

Note that compactness assumption in Theorem \ref{thm: result of kawai-yamamoto} (2) is indeed a necessary condition as our examples in Section \ref{section: main} shall demonstrate. Next, in the same spirit as \cite[Example 3.4]{Lotay2020}, we describe the simplest examples of deformed instantons in our context.

\textbf{Elementary examples.}
Let $M$ be a hyperK\"ahler manifold as above and suppose that $[c \om_i]\in H^2(M,\mathbb{Z})$ for some constant $c\in \R$. Then this defines a line bundle $L$ on $M$ and there exists a unitary connection $A$ such that $F_A=c  \om_i$. It is easy to see from (\ref{equ: dhym n=4}) that such a connection $A$ is an $\om_i$-dHYM connection with phase $1$ if and only if $c=0,\pm 1$. From Proposition \ref{prop: spin7modulesasSp2} we also see that $A$ is a deformed $\spin(7)$-instanton with respect to $\Phi_i$ for any $c$. In particular, together with Proposition \ref{prop: spin7modulesasSp2} we deduce that $A$ is both a $\spin(7)$-instanton and deformed $\spin(7)$-instanton with respect to $\Phi_i$ for any $c$. This shows that there is a non-trivial overlap of $\spin(7)$-instantons and deformed $\spin(7)$-instantons on Calabi Yau $4$-folds. On the other hand, $A$ is a deformed $\spin(7)$-instanton with respect to $\Phi_j$ or $\Phi_k$ if and only if $c=0,\pm 1$ but is a $\spin(7)$-instanton if and only if $c=0$. 

Note that it is not generally true that if $A$ is a (traceless) HYM connection then it is dHYM connection, likewise a $\spin(7)$-instanton is not generally a deformed $\spin(7)$-instanton. However, we do have that:
\begin{Prop}\label{prop: hyperholomorphic implies deformed}
	If $A$ is a hyper-holomorphic connection then $A$ is a $\om_i$-dHYM connection with phase $1$ for $i=1,2,3$.
\end{Prop}
\begin{proof}
	By assumption we have that $F_A \in \Lm^2_{10}\cong  \mathfrak{sp}(2)$ and hence it follows that $F_A \w \om^3_i=0$. 
	On the other hand we have that the $\Sp(2)$ module  $S^3(\mathfrak{sp}(2))$ i.e. the $3$-symmetric product of $\mathfrak{sp}(2)$, consists of $5$ irreducible modules, $2$ of which correspond to $\mathfrak{sp}(2)$ while the remaining $3$ modules have dimensions $35$, $81$ and $84$. Hence we deduce that $*F_A^3\in \Lm^2_{10}$ and in particular, $F_A^3 \w \om_i=0$. Thus, we have shown that (\ref{equ: dhym n=4}) holds with $\theta=0$ and this concludes the proof. 
\end{proof}
In particular, Proposition \ref{prop: hyperholomorphic implies deformed} also implies that hyper-holomorphic connections are both $\spin(7)$- and deformed $\spin(7)$-instantons. In other words, hyper-holomorphic connections lie in the intersection of all these special classes of connections. With these preliminaries in mind we next proceed to consider the dHYM equations on $T^*\C\mathbb{P}^1$.

\section{Deformed Hermitian Yang-Mills connections on $T^*\C\mathbb{P}^1$}\label{section: tcp1}

In this section we construct dHYM connections on the cotangent bundle of $\C\mathbb{P}^1$ endowed with the Eguchi-Hanson hyperK\"ahler structure \cite{EguchiHanson}. This set up provides a lower dimensional model of the situation we shall investigate in Section \ref{section: main} and hence it will serve to illustrate certain key similarities and differences.
It is well-known that the Eguchi-Hanson metric is invariant under a cohomogeneity one action of $\SU(2)$. Hence the problem of constructing $\SU(2)$-invariant dHYM connections reduces  to solving a non-linear system of ODEs. We shall show that the solutions to these ODEs are generally implicitly given by certain polynomials of degree equal to $n$, where $n$ is as in (\ref{equ: general dhym equation}), and the coefficients are functions of the geodesic distance. This feature will also occur in the Section \ref{section: main} but the equations will be more involved since in that case $n=4$. It is worth pointing out that such a phenomenon was also observed in \cite{AdamJacob2022} whereby the authors construct dHYM connections on the blow up of $\C\mathbb{P}^n$ (which in contrast to our set up is a compact manifold but similar to our set up also admits a cohomogeneity one action). 

Let $\eta_i$ denote the standard left invariant coframing of $S^3=\SU(2)$ satisfying $d\eta_i=\sum_{j,k}\varepsilon_{ijk}\eta_{jk}/2$, where $\varepsilon_{ijk}$ denotes the Levi-Civita symbol and $\eta_{jk}$ is shorthand for $\eta_j \w\eta_k$. The Eguchi-Hanson metric on $T^*\C\mathbb{P}^2\cong \mathcal{O}_{\C\mathbb{P}^1}(-2)$ can then be expressed as
\begin{equation}
g_{EH}=\Big(1-\frac{c}{r^4}\Big)^{-1} dr^2 + \frac{r^2}{4}\Big(1-\frac{c}{r^4}\Big) \eta_1^2+ \frac{r^2}{4} (\eta_{2}^2+\eta_{3}^2),\label{equ: eguchi hanson metric}
\end{equation}
where $c$ is a positive constant determining the size of the zero section $\C\mathbb{P}^1$ and $r\in [c^{1/4},\infty)$. Note that when $c=0$ the above metric corresponds to the flat metric on $\C^2/\mathbb{Z}_2$ and the Eguchi-Hanson space can be viewed as the resolution of this orbifold singularity obtained by blowing up the origin.

The associated hyperK\"ahler triple to the Eguchi-Hanson metric are given by
\begin{gather}
	\om_1 = \frac{1}{2} r dr \w \eta_1+\frac{r^2}{4}\eta_{23},\label{equ: eguchi hanson om1}\\
	\om_2 = \frac{r}{2}\Big(1-\frac{c}{r^4}\Big)^{-1/2}dr \w \eta_2+\frac{r^2}{4}\Big(1-\frac{c}{r^4}\Big)^{1/2}
	 \eta_{31},\label{equ: eguchi hanson om2}\\
	\om_3 = \frac{r}{2}\Big(1-\frac{c}{r^4}\Big)^{-1/2}dr \w \eta_3+\frac{r^2}{4}\Big(1-\frac{c}{r^4}\Big)^{1/2}\eta_{12}.\label{equ: eguchi hanson om3}
\end{gather}
One can indeed easily verify that $\nabla\om_i=0$, or equivalently, that $d\om_i=0$ for $i=1,2,3$. Note that the K\"ahler form $\om_1$ generates the cohomology group $H^2(T^*\C\mathbb{P}^1,\mathbb{Z})\cong \mathbb{Z}$; in particular, $\om_1$ is not exact. By contrast, we have that $\om_2=d({\sqrt{r^4-c}}{}\ \eta_2/4)$ and $\om_3=d({\sqrt{r^4-c}}{}\ \eta_3/4)$. The principal orbits of the $\SU(2)$ action on $T^*\C\mathbb{P}^1$ are all diffeomorphic to $\SO(3)\cong \SU(2)/\mathbb{Z}_2\cong \R\mathbb{P}^3$ while the singular orbit is $S^2\cong \SU(2)/\U(1)$ (this corresponds to the zero section of $T^*\C\mathbb{P}^1$ at $r=c^{1/4}$). Strictly speaking, since $\SO(3)$ is the $\mathbb{Z}_2$ quotient of $\SU(2)$, we shall also consider $\eta_i$ as a global coframing on $\SO(3)$ by suitably truncating the range of one of the Euler angles defining $\eta_i$ on $S^3$ cf. \cite[Section IIC]{EguchiHanson}; this is indeed implicit 
in the expressions (\ref{equ: eguchi hanson metric})-(\ref{equ: eguchi hanson om3}). 

A general $\SU(2)$-invariant abelian connection $1$-form over the open set $ \mathcal{O}_{\C\mathbb{P}^1}(-2)\backslash\C\mathbb{P}^1 \cong \R_{r>c^{1/4}} \times \SO(3)$ can be expressed  as
\begin{equation}
A= \sum_{i=1}^3f_i(r)\eta_i,\label{equ: ansatz for A on Eguchi-Hanson}
\end{equation}
where $f_i$ are arbitrary functions of $r\in (c^{1/4},\infty).$ 
Note that here we are expressing $A$ without a `$f_0(r)dr$' term since we can always set $f_0=0$ by a suitable gauge transformation; in other words we are expressing the connection $A$ in temporal gauge. A simple computation now shows that:
\begin{Prop}
The curvature form of $A$, as defined by (\ref{equ: ansatz for A on Eguchi-Hanson}), is given by
\begin{equation}
	F_A = dr \w (f_1'\eta_1+f_2'\eta_2+f_3'\eta_3)+f_1\eta_{23}+f_2\eta_{31}+f_3\eta_{12}.\label{equ: curvature eguchihanson}
\end{equation}
\end{Prop}
Since $H^2(\SO(3),\mathbb{Z})\cong \mathbb{Z}_2$ there are only $2$ complex line bundles over $\SO(3)$: the trivial one and a non-trivial one (which squares to the trivial one). The connection form $A$ is defined over the former bundle when $\eta_i$ are viewed as a coframing of $\SO(3)$ and over the latter bundle when $\eta_i$ are viewed as a coframing of the double cover $S^3$ instead.
In any case,
the space of line bundles over $T^*\C\mathbb{P}^1$ are parametrised by $H^2(T^*\C\mathbb{P}^1,\mathbb{Z})\cong H^2(\C\mathbb{P}^1,\mathbb{Z})\cong \mathbb{Z}.$ As can easily be seen from (\ref{equ: eguchi hanson om1})
the latter cohomology class is generated by the Fubini-Study form of the zero section which, without loss of generality, we can identify with  $k[\eta_{23}]=k[d\eta_1]$, where $k\in \mathbb{Z}$. 
Thus, in order for $A(r)$ to extend to a connection across the zero section $\C\mathbb{P}^1$, we need that $f_1(c^{1/4})=k$ and $f_2(c^{1/4})=f_3(c^{1/4})=0$, in which case the bundle over this $\C\mathbb{P}^1$ will be diffeomorphic to $\mathcal{O}_{\C\mathbb{P}^1}(k).$ With this in mind we can now proceed to the problem of constructing instantons on $T^*\C\mathbb{P}^1$.

The dHYM equations (\ref{equ: general dhym equation holo})-(\ref{equ: general dhym equation}) with respect to the K\"ahler form $\om_i$ are given by the pair
\begin{gather}
F_A \w \om_j = F_A \w \om_k=0,\label{equ: dhym n=2 1}\\
2 F_A \w \om_i = \tan(\theta) (\om_i^2-F_A^2),\label{equ: dhym n=2 2}
\end{gather}
where $(i,j,k)$ denotes as usual cyclic permutations of $(1,2,3)$. Note that on a \textit{compact} K\"ahler surface, Jacob-Yau showed that the existence of dHYM connections are equivalent a certain notion of stability on $L$ cf. \cite[Theorem 1.2]{JacobYau2017}, see also \cite{Chen2021}; however, we cannot appeal to their result since $T^*\C\mathbb{P}^1$ is a non-compact manifold. 
From (\ref{equ: dhym n=2 1}) and (\ref{equ: dhym n=2 2}) we see that when $\tan(\theta)=0$, the dHYM equation is equivalent to asking that $F_A$ is the curvature form of a hyper-holomorphic line bundle i.e. $F_A \w \om_i =0$ for $i=1,2,3$, or equivalently, $F_A \in \Lm^2_-\cong \mathfrak{su}(2)$ i.e. $A$ is an anti-self-dual instanton. For our purpose, it is worth pointing out that there are many such explicit examples on non-compact hyperK\"ahler manifolds owing to the results of Hitchin in \cite{Hitchin2014hyperholomorphic} (see also Remark \ref{hitchin permuting action} below). Note that the fact that phase $1$ dHYM connections coincide with traceless HYM connections only occurs in dimension $4$.

For the sake of comparison with the usual HYM condition we first begin by describing all such examples in our set up.
\begin{Prop}\label{prop: hym eguchi-hanson}\
Let $C_1,C_2, C_3$ be real constants, then on the open set $\mathcal{O}_{\C\mathbb{P}^1}(-2)\backslash\C\mathbb{P}^1 \cong \R_{r>c^{1/4}}\times \SO(3)$ we have that $A$ as given by (\ref{equ: ansatz for A on Eguchi-Hanson}) is a HYM connection with respect to
	\begin{enumerate}
		\item $\om_1$ such that $F_A \w \om_1 = \lambda \om_1 \w \om_1$ if and only if
		\[
		A = \frac{\lm r^4+C_1}{4 r^2}\eta_1+\frac{C_2}{\sqrt{r^4-c}}\eta_2+\frac{C_3}{\sqrt{r^4-c}}\eta_3.\]
		In particular, if $C_1=C_2=C_3=0$ then $F_A = \lm \om_1$.
		\item $\om_2$ such that $F_A \w \om_2 = \lambda \om_1 \w \om_1$ if and only if
		\[
		A = \frac{C_1}{r^2}\eta_1+\frac{\lm (r^4-c)+C_2}{4\sqrt{r^4-c}}\eta_2+\frac{C_3}{\sqrt{r^4-c}}\eta_3.\]
		In particular, if $C_1=C_2=C_3=0$ then $F_A = \lm \om_2$.
		\item $\om_3$ such that $F_A \w \om_3 = \lambda \om_1 \w \om_1$ if and only if
		\[
		A = \frac{C_1}{r^2}\eta_1+\frac{C_2}{\sqrt{r^4-c}}\eta_2+\frac{\lm( r^4-c)+C_3}{4\sqrt{r^4-c}}\eta_3.\]
		In particular, if $C_1=C_2=C_3=0$ then $F_A = \lm \om_3$.
		\end{enumerate}
	When $\lm=0$, $A$ is a hyper-holomorphic connection. If $C_2$ or $C_3$ is non-zero, the above connections do not extend smoothly across the zero section $\C\mathbb{P}^1$ but instead these correspond to meromorphic HYM connections which blow up along the zero section of $T^*\C\mathbb{P}^1$. Thus, there is (up to scaling) a unique globally well-defined hyper-holomorphic connection on $T^*\C\mathbb{P}^1$, namely $A=\eta_1/r^2$.
\end{Prop}
\begin{proof}
	The first part is a straightforward computation using (\ref{equ: curvature eguchihanson}) and (\ref{equ: dhym n=2 1}),(\ref{equ: dhym n=2 2}) with $\theta=0$; we refer the reader to the proof of Proposition \ref{prop: dhym on open set EH} below for the details. As discussed above the locally defined connection $A$ on $\mathcal{O}_{\C\mathbb{P}^1}(-2)\backslash\C\mathbb{P}^1$ extends to $T^*\C\mathbb{P}^2$ when the coefficient functions of $\eta_{2}$ and $\eta_{3}$ vanish at $r=c^{1/4}$ and this proves the last statement. 
\end{proof}
\begin{Rem}\label{hitchin permuting action}
In \cite[Sect. 2.2]{Hitchin2014hyperholomorphic} Hitchin showed that the aforementioned globally well-defined connection is in the fact the unique $S^1$-invariant hyper-holomorphic connection on $T^*\C\mathbb{P}^1$ which coincides with the Fubini-Study form on the zero section. Here the $S^1$ action is generated by the left-invariant vector field $X_1$ which is dual to $\eta_1$ i.e. $\eta_i(X_1)=\delta_{i1}$. This Killing circle action is permuting i.e. $\mathcal{L}_{X_1}\om_1=0$ but $\mathcal{L}_{X_1}\om_2=-\om_3$ and $\mathcal{L}_{X_1}\om_3=+\om_2$. Such examples of hyper-holomorphic connections arising from permuting Killing vector fields were first constructed by Haydys in \cite{Haydys2008}. In particular, this applies to any hyperK\"ahler manifold of the form $T^*(G/H)$, where $G/H$ is a Hermitian symmetric space and the permuting $S^1$ action corresponds to a rotation in the fibres. We shall again encounter this hyper-holomorphic connection in the next section where we shall have that $G/H=\SU(3)/\U(2)\cong \C\mathbb{P}^2$.
\end{Rem}
Next we consider the dHYM equations.  
\begin{Prop}\label{prop: dhym on open set EH}\
	Let $C_1,C_2, C_3$ be real constants, then on the open set $\mathcal{O}_{\C\mathbb{P}^1}(-2)\backslash\C\mathbb{P}^1\cong \R_{r>c^{1/4}}\times\SO(3)$ the connection form $A$ as given by (\ref{equ: ansatz for A on Eguchi-Hanson}) is a dHYM connection with respect to
	\begin{enumerate}
	\item $\om_1$ if and only if
	$f_2(r)=\frac{C_2}{\sqrt{r^4-c}}$, $f_3(r)=\frac{C_3}{\sqrt{r^4-c}}$
	and $f_1(r)$ is implicitly defined by
	\begin{equation}
	2\tan(\theta)f^2_1+r^2f_1-\frac{1}{8}\tan(\theta)\Big(r^4-\frac{16(C_2^2+C_3^2)}{r^4-c}\Big)=C_1.\label{equ: solution1 dhym eguchihanson}
	\end{equation}
	\item $\om_2$ if and only if
	$f_1(r)=\frac{C_1}{r^2}$, $f_3(r)=\frac{C_3}{\sqrt{r^4-c}}$
	and $f_2(r)$ is implicitly defined by
	\begin{equation}
	2\tan(\theta)f^2_2+\sqrt{r^4-c}f_2+2\tan(\theta)\Big(\frac{C_1^2}{r^4}+\frac{C_3^2}{r^4-c}-\frac{r^4}{16}\Big)=C_2.\label{equ: solution2 dhym eguchihanson}
	\end{equation}
	\item $\om_3$ if and only if
	$f_1(r)=\frac{C_1}{r^2}$, $f_2(r)=\frac{C_2}{\sqrt{r^4-c}}$
	and $f_3(r)$ is implicitly defined by
	\begin{equation}
	2\tan(\theta)f^2_3+\sqrt{r^4-c}f_3+2\tan(\theta)\Big(\frac{C_1^2}{r^4}+\frac{C_2^2}{r^4-C}-\frac{r^4}{16}\Big)=C_3.\label{equ: solution3 dhym eguchihanson}
	\end{equation}
\end{enumerate}
\end{Prop}
\begin{proof}
We begin with case $\mathrm{(1)}$. A direct computation using (\ref{equ: eguchi hanson om1})-(\ref{equ: eguchi hanson om3}) and (\ref{equ: curvature eguchihanson}) shows that $F_A\w \om_2=F_A \w \om_3=0$ is equivalent to 
\[(r^4-c)f_i'=-2r^3f_i,\]
for $i=2,3$ and solving the latter we get $f_2(r)=\frac{C_2}{\sqrt{r^4-c}}$, $f_3(r)=\frac{C_3}{\sqrt{r^4-c}}$. Using  (\ref{equ: eguchi hanson om1}) and (\ref{equ: curvature eguchihanson}), we find that
\[F_A \w \om_1 = \frac{2}{r^2}(rf_1'+2f_1)\ \om_1 \w \om_1\]
and 
\[\om_1^2-F_A^2=(-\frac{8f_1f_1'}{r^3}+(1+16\frac{C_2^2+C_3^2}{(r^4-c)^2}))\ \om_1 \w \om_1.\]
Thus, it follows that (\ref{equ: dhym n=2 2}) is given by the non-linear ODE:
\[2(r^2+4f_1 \tan(\theta))f_1'+4rf_1-r^3\Big(1+16\frac{C_2^2+C_3^2}{(r^4-c)^2}\Big)\tan(\theta)=0.\]
A close inspection of the latter shows that it is in fact an exact ODE. Indeed the reader will find no difficulty in verifying that differentiating (\ref{equ: solution1 dhym eguchihanson}) yields the latter ODE. Thus, this proves the first part. 

For case $(2)$, by an analogous calculation as above we find that $F_A \w \om_1=F_A \w \om_3 =0 $ if and only if  $f_1=\frac{C_1}{r^2}$ and $f_3=\frac{C_3}{\sqrt{r^4}-c}$. On the other hand from (\ref{equ: dhym n=2 2}) we find that $f_2(r)$ has to solve the non-linear ODE:
\[(4\tan(\theta)f_2+\sqrt{r^4-c})f_2'+\frac{2r^3}{\sqrt{r^4-c}}f_2-2\Big(\frac{4 C_1^2}{r^5}+\frac{4C_3^2 r^3}{(r^4-c)^2}+\frac{r^3}{4}\Big)\tan(\theta)=0.\]
Again it is not hard to see that the latter is an exact ODE and that the general solution is given by (\ref{equ: solution2 dhym eguchihanson}). Case $(3)$ is identical by swapping $f_2$ and $f_3$, and this concludes the proof.
\end{proof}
Since the $\tan(\theta)=0$ case is already covered by Proposition \ref{prop: hym eguchi-hanson} we shall now restrict to the case when $\tan(\theta)\neq 0$. Observe that in this case the solutions $f_i$ are implicitly defined as the solutions to certain quadratic polynomials (and that these reduce to linear equations precisely when $\tan(\theta)=0$). 
In order to determine which of the above locally defined dHYM connections extend to $T^*\C\mathbb{P}^1$, we need to impose the boundary conditions to the functions $f_i(r)$ at $r=c^{1/4}$. Moreover, we also need to verify that these are indeed well-defined for all values of $r\in[c^{1/4},\infty)$ since this is not a priori obvious from Proposition \ref{prop: dhym on open set EH}.
\begin{Th}\label{thm: main theorem TCP1}
Let $k\in \mathbb{Z}$ and suppose that $\tan(\theta) \neq 0$, then on $T^*\C\mathbb{P}^1$ the connection form $A$ as given by (\ref{equ: ansatz for A on Eguchi-Hanson}) is a dHYM connection with respect to
\begin{enumerate}
	\item $\om_1$ if and only if
	$f_2(r)=f_3(r)=0$
	and 
\begin{equation}
	f_1(r)=\frac{-r^2\pm\sqrt{\sec^2(\theta)r^4+\tan^2(\theta)(16k^2-c)+8kc^{1/2}\tan(\theta)}}{4\tan(\theta)},\label{solution1dhymTCP1}
\end{equation}
where we only take $+(-)$ if $4k\tan(\theta)+c^{1/2}>(<)0$. 
	\item $\om_2$ if and only if
$f_1(r)=\frac{c^{1/2}k}{r^2}$, $f_3(r)=0$
and 
\begin{equation}
	f_2(r)=\sqrt{r^4-c}\ \Big(\frac{-r^2\pm\sqrt{\sec^2(\theta)r^4+16k^2\tan^2(\theta)}}{4r^2\tan(\theta)}\Big).\label{solution2dhymTCP1}
\end{equation}
	\item $\om_3$ if and only if
$f_1(r)=\frac{c^{1/2}k}{r^2}$, $f_2(r)=0$
and 
\begin{equation}
	f_3(r)=\sqrt{r^4-c}\ \Big(\frac{-r^2\pm\sqrt{\sec^2(\theta)r^4+16k^2\tan^2(\theta)}}{4r^2\tan(\theta)}\Big).\label{solution3dhymTCP1}
\end{equation}
\end{enumerate}
\end{Th}
\begin{proof}
	Imposing the conditions that $f_2(c^{1/4})=f_3(c^{1/4})=0$ and $f_1(c^{1/4})=k$ in solution $(1)$ of Proposition \ref{prop: dhym on open set EH} gives that $C_2=C_3=0$ and
	\[C_1=2\tan(\theta) k^2+c^{1/2}k-\frac{c}{8}\tan(\theta).\]
	With $C_1$ as above and assuming $\tan(\theta)\neq 0$, we solve for $f_1$ in (\ref{equ: solution1 dhym eguchihanson}) to find (\ref{solution1dhymTCP1}).
	It is not hard to see from (\ref{solution1dhymTCP1}) that $f_1(r)$ is indeed well-defined for all values of $r\in[c^{1/4},\infty)$. 
	Next we consider case $(2)$.
	
	Imposing the same boundary conditions as before in solution $(2)$ of Proposition \ref{prop: dhym on open set EH} we now get that $C_1=c^{1/2}k$, $C_2=(16k^2-c)\tan(\theta)/8$ and $C_3=0$. Assuming that $\tan(\theta)\neq 0$, we solve for $f_2$ in (\ref{equ: solution2 dhym eguchihanson}) to find (\ref{solution2dhymTCP1}). 	Again it is clear that $f_2(r)$ is well-defined for all values of $r\in[c^{1/4},\infty)$. 
	The same argument applies for case $(3)$ and this concludes the proof.
\end{proof}
\begin{Rem}\
	\begin{enumerate}
		\item In case $(1)$ of Theorem \ref{thm: main theorem TCP1}, for generic values of $k$ and $\theta$ there is exactly one dHYM connection. Indeed when $r=c^{1/4}$, the term under the square root is given by $(4k\tan(\theta)+c^{1/2})^2$ and thus, we see that there is exactly one solution (among the two defined by (\ref{solution1dhymTCP1})) which satisfies $f_1(c^{1/4})=k$, except if $4k\tan(\theta)+c^{1/2}=0$ in which case we get $2$ distinct solutions. 
		It is also worth pointing out that by taking $\tan(\theta)=\frac{8 k c^{1/2}}{c-16 k^2}$ we recover the HYM solution $f_1=\lm r^2$ of Proposition \ref{prop: hym eguchi-hanson} $\mathrm{(1)}$ with $\lm=k/c^{1/2}, -c^{1/2}/16k$. 
		
		\item By contrast to case $(1)$, in cases $(2)$ and $(3)$ for generic values of $k$ and $\theta$ there are two dHYM connections. It is also worth pointing out that by taking $k=0$ we recover the HYM solution $f_2=\lm \sqrt{r^4-c}$ as in Proposition \ref{prop: hym eguchi-hanson} $\mathrm{(2)}$ i.e. $F_A$ is just a multiple of $\om_2$, with $\lm=(-1\pm |\sec(\theta)|)/4\tan(\theta)$.
		The analogous statement holds for case $(3)$.
	\end{enumerate}
\end{Rem}
Our results in this section show that although each K\"ahler form $\om_i$ is compatible with the same metric $g_{EH}$, the associated dHYM connections are rather different. In particular, for fixed value of phase angle $\theta$ and $k$ i.e. fixing the line bundle, the number of dHYM connections can be used to distinguish $\om_1$ from $\om_2$ and $\om_3$. We shall show in Section \ref{section: main} that a similar result also holds on $T^*\C\mathbb{P}^2$. 
On a compact manifold the phase angle $\theta$ is determined by the line bundle $L$ itself, so this freedom of varying both $\theta$ and $k$ is only occurs on non-compact manifolds; we illustrate this difference in an explicit example in next section.

\section{Deformed Hermitian Yang-Mills connections on $\SU(3)/\mathbb{T}^2$}\label{section: flag manifold}
In this section we construct $\SU(3)$-invariant dHYM connections on the flag manifold $F_{1,2}:=\SU(3)/\mathbb{T}^2$ with its standard homogeneous K\"ahler Einstein structure. This section serves two main purposes. Firstly it illustrates certain key differences between dHYM connections in the compact and non-compact setting, and secondly it will provide the basic set up for the cohomogeneity one $\SU(3)$-invariant connections that we shall investigate in the next section on $T^*\C\mathbb{P}^2$.

We begin by expressing the Maurer-Cartan form of $\SU(3)$ as
\[\Theta:=\begin{pmatrix}
	i(\theta_1+\theta_2) & i\theta_3-\theta_4 & \theta_5+i\theta_6 \\
	i\theta_3+\theta_4  & i(\theta_1-\theta_2) & i\theta_7+\theta_8 \\
	-\theta_5+i\theta_6  & i\theta_7-\theta_8 & -2i\theta_1 
\end{pmatrix}.\]
From the latter, the structure equations can be easily computed using $d\Theta+\Theta\w\Theta=0$. Since we shall use these frequently in the calculations throughout this article, we state them explicitly for the reader's convenience:
\begin{gather*}
	d\theta_{1}=-\theta_{56}+\theta_{78},\\
	d\theta_{2}=-2\theta_{34}-\theta_{56}-\theta_{78},\\
	d\theta_{3}=2\theta_{24}-\theta_{57}+\theta_{68},\\
	d\theta_{4}=-2\theta_{23}-\theta_{58}-\theta_{67},\\
	d\theta_{5}=3\theta_{16}+\theta_{26}+\theta_{37}+\theta_{48},\\
	d\theta_{6}=-3\theta_{15}-\theta_{25}-\theta_{38}+\theta_{47},\\
	d\theta_{7}=-3\theta_{18}+\theta_{28}-\theta_{35}-\theta_{46},\\
	d\theta_{8}=3\theta_{17}-\theta_{27}+\theta_{36}-\theta_{45},
\end{gather*}
where as before $\theta_{i...j}$ is shorthand for $\theta_i \w \cdots \w \theta_j$.
In what follows, we denote by $e_i$ the dual left invariant vector field to $\theta_i$ i.e. $\theta_i(e_j)=\delta_{ij}$.

Recall that the Hopf fibration is given by
\[\U(1) \hookrightarrow S^5 = \frac{\SU(3)}{\SU(2)}  \to \C\mathbb{P}^2 = \frac{\SU(3)}{\U(1)\SU(2)},\]
where the $\U(1)$ action is generated by $e_1$ and the $\SU(2)$ action is generated by $\langle e_2,e_3,e_4\rangle$. The Fubini-Study K\"ahler form of the base $\C\mathbb{P}^2$ 
can be identified with 
$\theta_{56}-\theta_{78}=:\om_{\C\mathbb{P}^2}$ (this corresponds to the curvature of the connection $1$-form $-\theta_1$).  
In view of twistor theory we also have the $S^2$ bundle  
$$\pi: F_{1,2}:=\SU(3)/\mathbb{T}^2 \to \C\mathbb{P}^2 
,$$ where the $\mathbb{T}^2$ action here is generated by $e_1$ and $e_3$; we refer the reader to \cite[Chap. 13]{Besse2008} and \cite[Chap. 4]{Salamon1989} for basic facts on twistor theory. We can express the K\"ahler Einstein structure of the flag manifold $F_{1,2}$ by
\begin{gather}
	g_{KE}= 2\theta_{2}^2+2\theta_{4}^2+\theta_{5}^2+\theta_{6}^2+\theta_{7}^2+\theta_{8}^2,\\
	\om_{KE}=d\theta_{3}=2\theta_{24}-\theta_{57}+\theta_{68}.\label{KE 2-form}
\end{gather}
We should point out that the twistor map $\pi$ is not holomorphic with respect to the natural complex structure of $\C\mathbb{P}^2$. Instead the twistor fibration can be viewed as the $2$-sphere bundle associated to the quaternionic K\"ahler structure of $\overline{\C\mathbb{P}^2}$ (which has the opposite orientation to the Fubini-Study one); indeed the quaternionic K\"ahler structure of $\overline{\C\mathbb{P}^2}$ is generated by the self-dual $2$-forms $\theta_{56}+\theta_{78}$, $\theta_{57}-\theta_{68}$ and $\theta_{58}+\theta_{67}$, and $F_{1,2}$ can be identified with its unit sphere bundle. Since $\om_{\C\mathbb{P}^2}$ is anti-self-dual with respect to the quaternionic K\"ahler structure it follows from twistor theory that $\pi^*\om_{\C\mathbb{P}^2}$ is of type $(1,1)$ with respect to the K\"ahler Einstein structure of $F_{1,2}$.

\begin{Rem}
	If we had considered the $\mathbb{T}^2$ action generated by $e_1$ and $e_2$ to define the flag manifold then we would have obtain an equivalent K\"ahler Einstein structure which is instead given by
	\begin{gather*}
		g_{}= 2\theta_{
			3}^2+2\theta_{4}^2+\theta_{5}^2+\theta_{6}^2+\theta_{7}^2+\theta_{8}^2,\\
		\om_{}=d\theta_{2}=-2\theta_{34}-\theta_{56}-\theta_{78}.
	\end{gather*}
	Similarly if one considers the $\mathbb{T}^2$ action generated by $e_1$ and $e_4$, we get another equivalent K\"ahler Einstein structure. This gives rise to the so-called triality. Since these are all related by a $\mathbb{Z}_3$-automorphism it suffices to restrict to only one case.
\end{Rem}
Before describing the $\SU(3)$-invariant abelian connections on $F_{1,2}$, we first note that $H^2(F_{1,2},\mathbb{Z})\cong \mathbb{Z}^2$ and hence the space of complex line bundles on $F_{1,2}$ is determined by a pair of integers. It is not hard to see that the general $\SU(3)$-invariant abelian connection $1$-form on the flag $F_{1,2}$ is given by
\[A=a_1\theta_{1}+a_3 \theta_{3},\] 
where $(a_1,a_3)\in \mathbb{Z}^2\cong H^2(F_{1,2},\mathbb{Z})$ i.e. $A$ corresponds to an $\SU(3)$-invariant connection on each of the distinct line bundles. Moreover, using the structure equations we have that
\begin{equation}
F_A=  a_1(-\theta_{56}+\theta_{78})+a_3(2\theta_{24}-\theta_{57}+\theta_{68})=-a_1 \pi^*\om_{\C\mathbb{P}^2}+a_3\om_{KE}\label{KE curvature}
\end{equation} 
and it follows from the above discussion that $A$ is a holomorphic connection with respect to the K\"ahler Einstein structure of $F_{1,2}$.

\begin{Rem}
On a general \textit{compact} K\"ahler manifold $M^{2n}$, the phase angle $\theta$ is determined by the Chern class of the line bundle itself, i.e. by $[F_A]$, via 
\begin{equation}
\theta=\arg(\int_M(\om+iF_A)^n).\label{equ: definition of phase angle compact}
\end{equation}
This is not generally applicable in the non-compact set up since the latter integral might be infinite; this was indeed the case in our cohomogeneity one framework in the last section. This explains why we had the freedom to vary both $k$ and $\theta$ independently in the Theorem \ref{thm: main theorem TCP1}. In the $F_{1,2}$ situation, however, it follows that $\theta$ is completely determined by the pair $(a_1,a_2)$. We shall now describe the dHYM connections in this case explicitly.
\end{Rem}
\begin{figure}
	\centering
	\includegraphics[height=7cm]{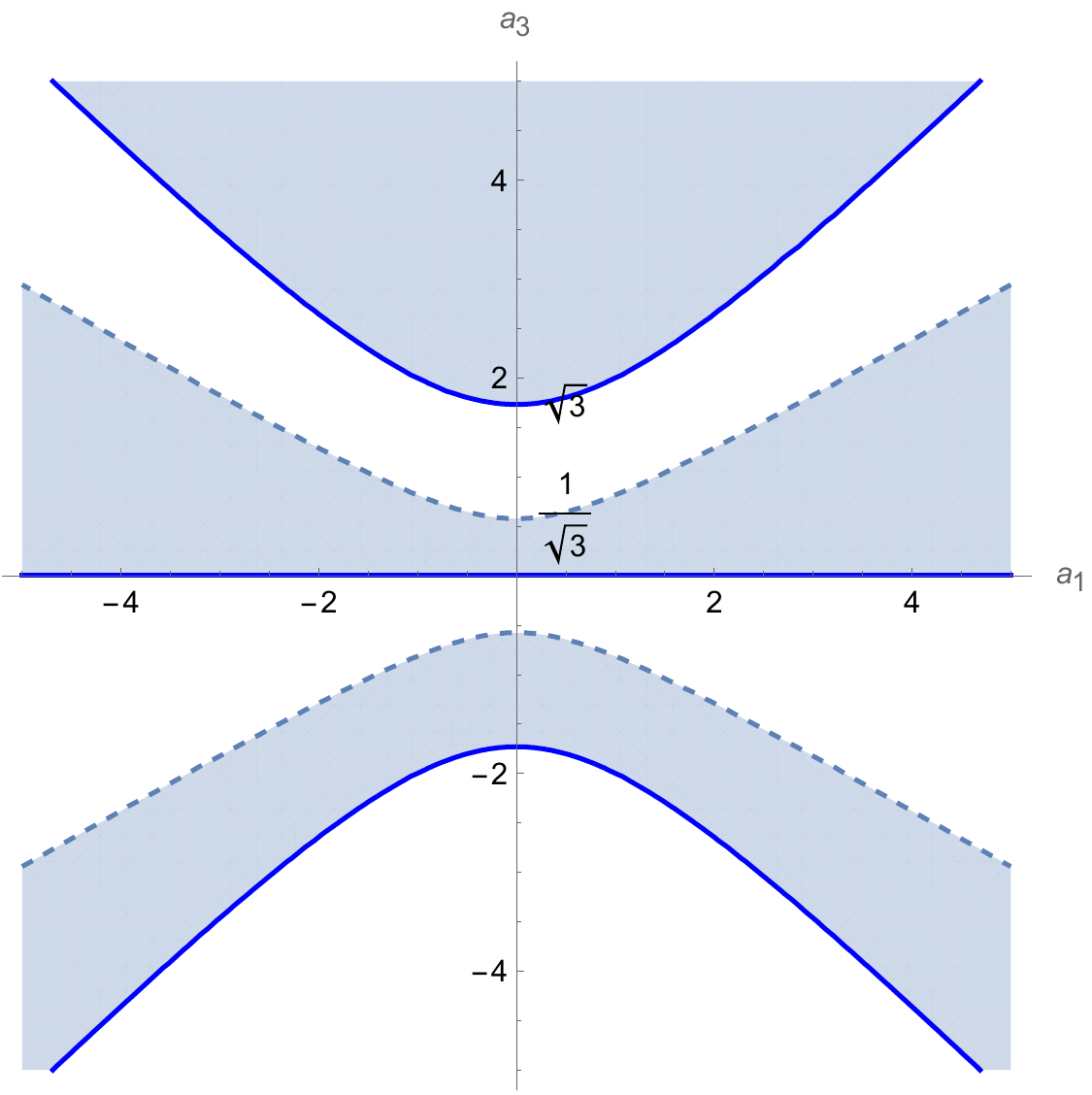}
	\caption{{{The shaded region corresponds to $\tan(\theta)\geq 0$, bold line to $\tan(\theta)=0$ and dotted line to $\cot(\theta)=0$} in (\ref{equ: f12 tantheta})}}\label{figure100}
\end{figure}
\begin{Th}\label{theorem: dhym on flag}
	The connection form $A=a_1\theta_{1}+a_3 \theta_{3}$ is a dHYM connection with respect to $\om_{KE}$ if and only if
	\begin{equation}
	\tan(\theta)=\frac{a_3(a_3^2-a_1^2-3)}{3a_3^2-a_1^2-1},\label{equ: f12 tantheta}
	\end{equation}	
or equivalently, 
\begin{equation}
\theta \equiv \arctan({a_3}) + \arctan(a_3+a_1)+\arctan(a_3-a_1) \mod 2\pi. \label{equ: f12 theta}
\end{equation}
\end{Th}
\begin{proof}
First setting $n=3$ in (\ref{equ: general dhym equation}) gives $$ 3 F_A \w \om_{KE}^2- F_A^3 = \tan(\theta) (\om_{KE}^3- 3 F_A^2 \w \om_{KE}).$$ 
Since we already know that $A$ is a holomorphic connection, it suffices to solve the latter equation. A straightforward computation using (\ref{KE 2-form}) and (\ref{KE curvature}) shows that
\begin{gather*}
	F_A \w \om_{KE}^2= 12 a_3 \theta_{245678},\\
	F_A^3= -12(a_1^2 a_3-a_3^3)\theta_{245678},\\
	\om_{KE}^3= 12\theta_{245678},\\
	F_A^2 \w \om_{KE} = -4(a_1^2-3a_3^2)\theta_{245678},
\end{gather*}
and from this we easily deduce (\ref{equ: f12 tantheta}). To obtain the second expression we follow the approach in \cite[Section 2]{Collins2018deformed}. First using $\om_{KE}$ we can view the $(1,1)$-form $F_A$ as an endomorphism of holomorphic tangent bundle $T^{1,0}$; concretely 
it is can expressed as
\[\begin{pmatrix}
	{a_3} & 0 &  0 \\
	0 & a_3 &  -a_1 \\
	0 & -a_1 &  a_3 
\end{pmatrix}\]
and has eigenvalues $a_3, a_3+a_1$ and $a_3-a_1$. Thus, we have that 
\[re^{i\theta}=
\frac{(\om_{KE}+iF_A)^3}{\om_{KE}^3}=(1+i{a_3})(1+i(a_3+a_1))(1+i(a_3-a_1)),
\]
for some positive number $r$, and using (\ref{equ: definition of phase angle compact}) we get (\ref{equ: f12 theta}).
\end{proof} 
In \cite{JacobYau2017} Jacob-Yau showed that solutions to the dHYM equations on a \textit{compact} K\"ahler manifold are unique in a given cohomology class. Hence it follows that the solutions in Theorem \ref{theorem: dhym on flag} in fact consist of all the dHYM connections on $F_{1,2}$ i.e. not just the invariant ones. Figure \ref{figure100} illustrates how the phase angle depends on the choice on the line bundle over $F_{1,2}$. 
Having given a concrete description of $\SU(3)$, next we proceed to our main goal, namely, the construction of  the deformed instantons equations on $T^*\C\mathbb{P}^2$. 

\section{Deformed Hermitian Yang-Mills and deformed $\spin(7)$ instantons on $T^*{\C\mathbb{P}^2}$}\label{section: main}

In this section we construct dHYM connections on the cotangent bundle of $\C\mathbb{P}^2$ endowed with the Calabi hyperK\"ahler structure \cite{CalabiAnsatz1979}. It is well-known that the latter is invariant under a cohomogeneity one action of $\SU(3)$ and hence the problem of constructing $\SU(3)$-invariant dHYM connections reduces to solving a non-linear system of ODEs. This is a higher dimensional analogue of the problem considered in Section \ref{section: tcp1} albeit more involved. As we already saw in Section \ref{section: preliminaries} one can also define `natural' $\spin(7)$-structures on a  hyperK\"ahler $8$-manifold by (\ref{sp2andspin7i}). Hence we shall also consider deformed $\spin(7)$ instantons in this set up and compare the two cases. 

We begin by describing the cohomogeneity one hyperK\"ahler structure of $T^*\C\mathbb{P}^2$ following the same notation as introduced in the previous section. First we recall the well-known fact that for each pair of integers $k,l$ there are different embeddings of $\U(1)$ in the maximal torus of $\SU(3)$. Concretely, we denote these by $\U(1)_{k,l}:=\mathrm{diag}(e^{ik\vp},e^{il\vp}, e^{-i(k+l)\vp})$, where $\vp \in [0,2\pi)$. Observe that the infinitesimal vector field associated to this $\U(1)_{k,l}$ action is given by $\frac{k+l}{2}e_1+\frac{k-l}{2}e_2$. The homogeneous spaces $N_{k,l}:=\SU(3)/\U(1)_{k,l}$ are called the Aloff-Wallach spaces. Here we shall only consider the case when $k=l=1$ since $N_{1,1}$ is precisely the principal orbit of the cohomogeneity one action of $\SU(3)$ on $T^*\C\mathbb{P}^2$ and for convenience we shall denote $\U(1)_{1,1}$ simply by $\U(1)$. On the other hand, the singular orbit of the $\SU(3)$ action is isomorphic to $\C\mathbb{P}^2\cong \SU(3)/\U(2)$, where the $\U(2)$ is generated by $\langle e_1,e_2,e_7,e_8\rangle$ (see Proposition \ref{prop: HK structure of TCP2} below). It is worth pointing out that in fact the Aloff Wallach spaces $N_{1,1}, N_{1,-2}$ and $N_{-2,1}$ are all diffeomorphic under the action of the Weyl group of $\SU(3)$; however, when viewed as bundles over the latter $\C\mathbb{P}^2$ they correspond to either an $\R\mathbb{P}^3$ or $S^3$ bundle and this subtlety turns out to be rather important as we shall see below. In \cite{Cvetivc2002} the authors show that the Calabi hyperK\"ahler structure in fact belongs to a $1$-parameter family of complete $\spin(7)$ metrics and moreover that there also exist $\SU(3)$-invariant $\spin(7)$ metrics with $N_{k,l}$ as the principal orbits for each $k,l$. In this paper we shall not consider these more general families since these are not explicit and hence more involved to study. 

For the homogeneous space $N_{1,1}$, we have the reductive splitting
\begin{gather*}
\mathfrak{su}(3)\cong \mathfrak{u}(1)\oplus \mathfrak{m},
\end{gather*}
where $\mathfrak{u}(1)\cong\langle e_1 \rangle$ and $\mathfrak{m}$ denotes its $\mathrm{Ad}\big|_{\U(1)}$ complement. As a $\U(1)$-module we have that
$$\mathfrak{m} \cong 3\R \oplus 2\C,$$ 
where the three copies of $\R$ are generated by $\langle e_2,e_3,e_4\rangle$ and the two copies of $\C$ are spanned by $\langle e_5, e_6 \rangle$  and $\langle e_8, e_7 \rangle$. Note that the $\U(1)$ acts with weight $3$ on each of these copies of $\C$ i.e. we have that $[e_1,e_5]=3e_6$,  $[e_1,e_6]=-3e_5$ and $[e_1,e_8]=3e_7$, $[e_1,e_7]=-3e_8$.  The tangent space of $N_{1,1}$ at a point can be identified with the isotropy representation $\mathfrak{m}$. It follows that the space of $\SU(3)$-invariant $2$-forms on $N_{1,1}$ can then be identified with $\U(1)$-invariant $2$-forms on $\mathfrak{m}$. A simple computation shows that as $\U(1)$-module 
$$\Lm^2(\mathfrak{m})\cong 7\R \oplus 6 \C$$ 
and hence there are $7$ such invariant $2$-forms: $3$ of them are obtained as combinations of $\langle\theta_2,\theta_3,\theta_4\rangle$ and the other $4$ are given by
$$\langle \theta_{56}, \theta_{78}, \theta_{57}+\theta_{86},\theta_{58}+\theta_{67}\rangle.$$ Using these invariant forms on $N_{1,1}$ we can now express the $\Sp(2)$-structure of $T^*\C\mathbb{P}^2$ as follows:
\begin{Prop}\label{prop: HK structure of TCP2}
	The Calabi hyperK\"ahler structure of $T^*{\C\mathbb{P}^2}$ is explicitly given by
	\begin{gather}
		g_{Ca} = f_0^2dr^2+f_2^2\theta^2_2+f_3^2(\theta_3^2+\theta_4^2)+
		f_5^2 (\theta_5^2+\theta_6^2)+f_7^2(\theta_7^2+\theta_8^2),\label{equ: tcp2 metric}\\
		\om_1 = f_0f_2dr \w \theta_2 +f_3^2 \theta_{34}+f_5^2\theta_{56}+f_7^2\theta_{78},\label{equ: tcp2 om1}\\
		\om_2 = f_0f_3dr \w \theta_3 +f_2f_3 \theta_{42}+f_5f_7\sigma_{2},\label{equ: tcp2 om2}\\
		\om_3 = f_0f_3dr \w \theta_4 +f_2f_3 \theta_{23}+f_5f_7\sigma_{3},\label{equ: tcp2 om3}
	\end{gather}	
	where $\sigma_{2}:=\theta_{57}+\theta_{86}$, $\sigma_{3}:=\theta_{58}+\theta_{67}$,
	\begin{gather*}
		f_0=-(1-4c^2r^{-4})^{-1/2},\\
		f_2=r(1-4c^2r^{-4})^{1/2}=\frac{2}{r}(r^2/2+c)^{1/2}(r^2/2-c)^{1/2},\\
		f_3=r,\\
		f_5=(r^2/2+c)^{1/2},\\
		f_7=(r^2/2-c)^{1/2},
	\end{gather*}
	and $c$ is a constant determining the size of the zero section $\C\mathbb{P}^2$. Without loss of generality, we shall assume that $c>0$ and take $r\in [\sqrt{2c},+\infty)$. The bolt (or zero section) in this case is a $\C\mathbb{P}^2$ spanned by $\langle e_3,e_4,e_5,e_6\rangle$ and where the Fubini-Study form is, up to a constant, given by $\theta_{34}+\theta_{56}$. 
\end{Prop}
\begin{proof}
	First we need to verify that $d\om_1=d\om_2=d\om_3=0.$ Consider the $\Sp(2)$-structure defined by (\ref{equ: tcp2 metric})-(\ref{equ: tcp2 om3}) whereby $f_i(r)$ are arbitrary functions. Then a straightforward calculation using the structure equations of $\SU(3)$ shows that $d\om_1=0$ is equivalent to
	\begin{gather*}
		f_3^2=f_5^2+f_7^2,\\
		f_0f_2=-\frac{1}{2}(f_3^2)'=-(f_5^2)'=-(f_7^2)'.
	\end{gather*}
	On the other hand, the condition that $d\om_2=0$ (which coincides with $d\om_3=0$) is equivalent to
	\begin{gather*}
		f_2f_3=2f_5f_7,\\
		f_0f_3=-\frac{1}{2}(f_2f_3)'=-(f_5f_7)'.
	\end{gather*}
	The functions $f_i(r)$ as given in the proposition follow immediately by solving the above system of ODEs. To conclude we still need to show that the metric $g_{Ca}$ extends smoothly across the zero section and hence is indeed complete. Since this was already demonstrated in \cite{CalabiAnsatz1979}, see also \cite{Cvetivc2002}, this concludes the proof.
\end{proof}
As already mentioned in Remark \ref{hitchin permuting action}, there exists again a permuting Killing vector field, given by $e_2/2$, such that $\mathcal{L}_{e_2}\om_1=0$, $\mathcal{L}_{e_2}\om_2=+2\om_3$ and $\mathcal{L}_{e_2}\om_3=-2\om_2$. As in Section \ref{section: tcp1} this will again give rise to a hyper-holomorphic connection (see Lemma \ref{lemma: holomorphic} below). 
Note that we can always set the constant $c=1$ by a suitably reparametrising the radial coordinate $r$, however we shall find it more useful to keep it so as to facilitate comparison with the conical situation whereby $c=0$, see Section \ref{section: bscone}. To compute the deformed $\spin(7)$ instanton equations later on we shall need the next result. 
\begin{Prop}\label{prop: spaces Ei}
With respect to the Calabi hyperK\"ahler structure, as given in Proposition \ref{prop: HK structure of TCP2}, the space of $\SU(3)$-invariant $2$-forms in each of the spaces $E_1, E_2,E_3$, as defined by (\ref{definition of Ei}), is one dimensional and these are generated by
\begin{gather}\label{prop: compute Eis}
\al_1=f_0f_2dr \w \theta_2 +f_3^2 \theta_{34}-f_5^2\theta_{56}-f_7^2\theta_{78},\\
\al_2=f_0f_3dr \w \theta_3 +f_2f_3 \theta_{42}-f_5f_7\sigma_{2},\\
\al_3=f_0f_3dr \w \theta_4 +f_2f_3 \theta_{23}-f_5f_7\sigma_{3},
\end{gather}
respectively. In particular, from (\ref{definition of F_i^+}) we have that $F_i^+=\langle \al_j \w \om_k \rangle$.
\end{Prop}
\begin{proof}
We already showed that a general $\SU(3)$-invariant $2$-form on $N_{1,1}$ is given by a linear combination of $\langle \theta_{23},\theta_{24},\theta_{34},\theta_{56},\theta_{78},\sigma_{2},\sigma_{3}\rangle$. Thus, it follows that a general $\SU(3)$-invariant $2$-form on $T^*\C\mathbb{P}^2$ consists of a combination of the latter seven $2$-forms together with the triple of $2$-forms $\langle dr \w \theta_2, dr \w \theta_3, dr \w \theta_4\rangle$. Using the definition of the spaces $E_i$ as given by (\ref{definition of Ei}), we find that in this ten dimensional family there is exactly one such $2$-form that lies in each of $E_i$.	
In fact in this ten dimensional family of $\SU(3)$-invariant $2$-forms, three correspond to $\om_1,\om_2,\om_3$, one lies in each of $E_i$, and the remaining four lie in $\Lm^2_{10}$. The last statement follows from (\ref{definition of F_i^+}).
\end{proof}

Observe that when $c=0$, the hyperK\"ahler metric (\ref{equ: tcp2 metric}) corresponds to the cone metric over $N_{1,1}$ with its well-known $3$-Sasakian structure cf. \cite{GalickiSalamon1996}. Galicki-Salamon showed in \cite[Proposition 2.4]{GalickiSalamon1996} that squashing a  $3$-Sasakian metric along its $3$-dimensional foliation (given in our case by $\langle e_2,e_3,e_4\rangle$) by a factor of $1/\sqrt{5}$ yields a strictly nearly parallel $\G_2$ metric: the latter is characterised by the fact the cone metric has holonomy \textit{equal} to $\spin(7)$, in contrast to the $3$-Sasakian case whereby the cone metric has holonomy \textit{equal} to $\Sp(2)$. For the squashed structure on $S^7$ the associated conical $\Spin(7)$ metric can in fact be smoothed to a complete metric on the spinor bundle of $S^4$: this is the so-called Bryant-Salamon metric cf. \cite[Sect. 4. Theorem 2]{Bryant1989}. In our set up, the analogous $\SU(3)$-invariant Bryant-Salamon $\spin(7)$ metric is not smooth but instead has orbifold singularities since $\C\mathbb{P}^2$ is not a spin manifold. Explicitly, it is given as follows: 
\begin{Prop}\label{prop: spin7 structure of TCP2}
	The orbi-spinor bundle of $\C\mathbb{P}^2$ admits the Bryant-Salamon $\spin(7)$-structure and is explicitly given by
	\begin{align}
		&g_{BS}=h_0^2dr^2+h_2^2(\theta_2^2+\theta_3^2+\theta_4^2)+h_5^2(\theta_5^2+\theta_6^2+\theta_7^2+\theta_8^2),\label{equ: bs metric}\\
		\Phi_{BS}=\ &h_0dr \w (h_2^3\theta_{234}-h_2 h_5^2(\theta_{256}+\theta_{278}+\theta_{357}-\theta_{368}+\theta_{458}+\theta_{467})-\label{BSspin74form}\\
		&h_2^2h_5^2(\theta_{2358}+\theta_{2367}-\theta_{2457}+\theta_{2468}+\theta_{3456}+\theta_{3478})+h_5^4\theta_{5678}\nonumber
	\end{align}
where
\begin{gather*}
	h_0=\frac{2}{(r^2+c)^{1/5}},\\
	h_2=\frac{2 r}{(r^2+c)^{1/5}},\\
	h_5=\sqrt{10} (r^2+c)^{3/10},
\end{gather*}
$r\in [0,+\infty)$ and $c$ is a non-negative constant determining the size of the zero section $\C\mathbb{P}^2$ spanned by $\theta_5,\theta_6,\theta_7,\theta_8$. When $c=0$ we get a cone metric over $N^{1,1}$ with its strictly nearly parallel $G_2$-structure, and when $c>0$ we get a metric which has a $(\R^4/\mathbb{Z}_2)$-orbifold singularity along the zero section.
\end{Prop}
\begin{proof}
	It suffices to verify that $d\Phi_{BS}=0$ which is a straightforward computation using the structure equations.
\end{proof}

Observe that the singular orbits in Proposition \ref{prop: HK structure of TCP2} and Proposition \ref{prop: spin7 structure of TCP2} correspond to two different $\C\mathbb{P}^2$s. In the former case the principal orbit $N_{1,1}$ corresponds to an $S^3$ bundle over this $\C\mathbb{P}^2$ whereas in the latter case it corresponds to an $\R\mathbb{P}^3$ bundle which explains why the metric $g_{BS}$ is not smooth but instead has orbifold singularities, see also \cite[Sect. 3.1.3]{Cvetivc2002} for a general description of $N_{k,l}$ as a lens space bundle over $\C\mathbb{P}^2$. For the sake of comparison we shall also describe the deformed $\spin(7)$ instantons for the hyperK\"ahler and $\spin(7)$ cones (see Section \ref{section: bscone}).

\subsubsection{Invariant connections}
To classify abelian $\SU(3)$-invariant connections on the principal orbit $N^{1,1}:=\SU(3)/\U(1)$ we appeal to Wang's theorem cf. \cite[Chap. 10 Th. 2.1]{KobayashiNomizu2}. In our context, Wang's theorem asserts that given a group homomorphism $\lambda_k:\U(1)\to \U(1)$ such that $e^{i\theta}\mapsto e^{ik\theta}$, all $\SU(3)$-invariant connections on the associated principal $\U(1)$-bundle
\[\SU(3)\times_{\lambda_k} \U(1) \to N^{1,1}\]
are classified by $\U(1)$-equivariant maps
\[\Lm: (3\R \oplus 2\C, \mathrm{Ad}\big|_{\U(1)})\to (\mathfrak{u}(1), \mathrm{Ad}\circ \lambda_k).\]
An easy application of Schur's lemma shows that there is a $3$-parameter family of such connections for each given $k\in\mathbb{Z}$. Concretely, we have that abelian $\SU(3)$-invariant connections on the principal orbits $\{r\}\times N^{1,1}$ are of the form
\begin{equation}
A(r)=k\theta_1+a_2(r) \theta_2+a_3(r) \theta_3+a_4(r) \theta_4.\label{equ: general connection A}
\end{equation}
where $k$ is a constant. We should point out that these connections were also described in \cite{BallOliveira} whereby the authors used them to construct $\G_2$-instantons on the Aloff-Wallach spaces $N_{k,l}$. Next we compute the curvature form as follows:
\begin{Prop}\label{prop: curvature FA}
	The curvature form of $A$, as defined by (\ref{equ: general connection A}), is given by 
\begin{align}
\label{equ: curvature tcp2}	F_A =\ & dr\w( a_2'\theta_2+a_3' \theta_3+a_4' \theta_4)-2a_2 \theta_{34}-(a_2+k) \theta_{56}-(a_2-k) \theta_{78}\\ 
	&+2a_3 \theta_{24}-a_3 \sigma_{2}-2a_4 \theta_{23}-a_4 \sigma_{3},\nonumber
\end{align}
where we recall that $\sigma_{2}:=\theta_{57}+\theta_{86}$ and $\sigma_{3}:=\theta_{58}+\theta_{67}$.
\end{Prop}
Observe that $F_A$ is indeed a well-defined $2$-form on the open set $\R_{r>\sqrt{2c}}\times N^{1,1}$. However, in order for $A(r)$ to extend to a connection on $T^*\C\mathbb{P}^2$ we need it to extend to a smooth connection on the zero section $\C\mathbb{P}^2$ as well. From Proposition \ref{prop: HK structure of TCP2} it is easy to see that the tangent space of the latter $\C\mathbb{P}^2$ is spanned by $\langle e_3,e_4, e_5,e_6\rangle$ and that  $\om_1\big|_{\C\mathbb{P}^2}=\sqrt{2c}(\theta_{34}+\theta_{56})$ corresponds precisely to its Fubini-Study form. Since $H^2(\C\mathbb{P}^2,\mathbb{Z})\cong \mathbb{Z}$ is generated by the cohomology class of the latter form, it follows from Chern-Weil theory that we need $A$ to extend to a connection $1$-form whose curvature form is an integer multiple of $\theta_{34}+\theta_{56}$ after suitably fixing the value of $c$. Inspecting the structure equations of $\SU(3)$, we see that 
\[d(\theta_1+\theta_2)=-2(\theta_{34}+\theta_{56}).\]
Thus, we deduce that the connection $A(r)$ of the form (\ref{equ: general connection A}) on $\R_{r>\sqrt{2c}}\times N^{1,1}$ will extend to a connection on $T^*\C\mathbb{P}^2$ provided that $a_3(\sqrt{2c})=a_4(\sqrt{2c})=0$ and $a_2(\sqrt{2c})=k\in \mathbb{Z}$.  In other words, we need that $A(\sqrt{2c})=k(\theta_{1}+\theta_{2})$. With this in mind, we can now proceed to the construction of instantons.

\subsubsection{{HYM connections and $\spin(7)$-instantons}}

We begin by describing the $\SU(3)$-invariant holomorphic connections on 
the open set $N_{1,1}\times \R_{r>\sqrt{2 c}}$.
\begin{Lemma}\label{lemma: holomorphic}
	Let $k \in \mathbb{Z}$, then the connection $A=k \theta_1+a_2\theta_2+a_3\theta_3+a_4\theta_4$
	is holomorphic with respect to 
	\begin{enumerate}
		\item $\om_1$ if and only if $a_3=a_4=0$. 
		\item $\om_2$ if and only if $a_2= 2ck/r^2$ and $a_4=0$.
		\item $\om_3$ if and only if $a_2= 2ck/r^2$ and $a_3=0$.
	\end{enumerate}
	In particular, $A$ is hyper-holomorphic i.e. $F_A \in  \Lm^2_{10}$ if and only if $a_2=2ck/r^2$ and $a_3=a_4=0$. 
\end{Lemma}
\begin{proof}
	Recall that $A$ is a holomorphic connection with respect to $\om_i$ if and only if $F_A \w (\om_j+i\om_k)^2=0$. The result follows easily by computing the latter using (\ref{equ: tcp2 om1})-(\ref{equ: tcp2 om3}) and (\ref{equ: curvature tcp2}). 	
\end{proof}
Observe when $r=\sqrt{2 c}$ we indeed have that the above hyper-holomorphic connection satisfies $A(\sqrt{2 c})=k(\theta_1+\theta_2)$ and hence it extends smoothly to a connection on $T^*\C\mathbb{P}^2$.
As we already mentioned in Remark \ref{hitchin permuting action}, this hyper-holomorphic connection is precisely the one described by Hitchin in \cite[Section 2.2]{Hitchin2014hyperholomorphic}. Moreover, Hitchin showed that this connection is characterised as the unique hyper-holomorphic connection on $T^*\C\mathbb{P}^2$ which is invariant under the $S^1$ action generated by the permuting Killing vector field $e_2/2$ and whose curvature form restricts to Fubini-Study form on the zero section $\CP^2$. 
Furthermore, it is easy to verify that its curvature form is given by
\[F_A = \om_1 + dd^c\mu,\]
where $\mu$ is the moment map defined by $d\mu=e_2 \ip \om_1/2$, which is consistent with \cite[Theorem 1]{Hitchin2014hyperholomorphic}.

Having described the holomorphic connections, next we investigate which among them correspond to HYM connections.
\begin{Prop}\label{prop: hym on tcp2}
	Let $k \in \mathbb{Z}$,  then the connection $A=k\theta_1 + a_2\theta_2 + a_3\theta_3 + a_4\theta_4$ is an
\begin{enumerate}
	\item $\om_1$-HYM connection on $N_{1,1}\times \R_{r>\sqrt{2 c}}$ such that $F_A \w \om_1^3 =\lm \om_1^4$ if and only if $a_3=a_4=0$ and
	\begin{equation}
		a_2=\frac{C_0-(r^4-4c^2)(\lambda (r^4-4c^2)-4ck)}{2 r^2(r^4-4c^2)},\label{hym om1}
	\end{equation}
	where $C_0$ is a constant. $A$ extends smoothly to $T^*\C\mathbb{P}^2$ only when $C_0=0$, in which case $F_A$ is the sum of $\lm\om_1$ and the curvature of the hyper-holomorphic connection.
	\item $\om_2$-HYM on $N_{1,1}\times \R_{r>\sqrt{2 c}}$ such that $F_A \w \om_2^3 =\lambda\om_2^4$ if and only if $a_2=2ck/r^2$, $a_4=0$ and 
	\begin{equation}
		a_3=\frac{C_0-\lambda(r^4-4c^2)^2}{2(r^4-4c^2)^{3/2}},\label{hym om2}
	\end{equation}
	where $C_0$ is a constant. $A$ extends smoothly to $T^*\C\mathbb{P}^2$ only when $C_0=0$, in which case $F_A$ is the sum of $\lm\om_2$ and the curvature of the hyper-holomorphic connection. 
	\item $\om_3$-HYM on $N_{1,1}\times \R_{r>\sqrt{2 c}}$ such that $F_A \w \om_2^3 =\lambda\om_3^4$ if and only if $a_2=2ck/r^2$, $a_3=0$ and 
	\begin{equation}
		a_4=\frac{C_0-\lambda(r^4-4c^2)^2}{2(r^4-4c^2)^{3/2}},
	\end{equation}
	where $C_0$ is a constant. $A$ extends smoothly to $T^*\C\mathbb{P}^2$ only when $C_0=0$, in which case $F_A$ is the sum of $\lm\om_3$ and the curvature of the hyper-holomorphic connection. 
	\end{enumerate}
\end{Prop}
\begin{proof}
In case (1), from Lemma \ref{lemma: holomorphic}, we already know that we need $a_3=a_4=0$. Using Propositions \ref{prop: HK structure of TCP2} and \ref{prop: curvature FA} we find that $F_A\w\om_1^3=\lm\om_1^4$ is equivalent to the
\[r^2a_2'(r^4-4c^2)+2ra_2(3r^4-4c^2)+4\lm r^7-8cr^3(k+2c\lm)=0.\]
The latter is a linear ODE and one easily solves it to get (\ref{hym om1}). 
For case (2), with $a_2=2ck/r^2$ and $a_3=0$, we have that $F_A\w\om_2^3=\lm\om_1^4$ is equivalent to
\[a_3'(r^4-4c^2)+6r^3 a_3+4\lm r^3\sqrt{r^4-4c^2}=0\]
and solving  the latter yields (\ref{hym om2}). 
For case (3) it suffices to swap $a_3$ and $a_4$ in case (2), and this concludes the proof. 
\end{proof}
Next we consider the $\spin(7)$-instantons with respect $\Phi_i$ as defined by (\ref{sp2andspin7i}).
\begin{Th}\
	Let $k\in \mathbb{Z}$ and $C_i \in \R$, then
	\begin{enumerate}
	\item the general $\SU(3)$-invariant abelian $\Spin(7)$-instanton on  $N_{1,1}\times \R_{r>\sqrt{2 c}}$ with respect to $\Phi_1$ is given by
	\begin{equation}
		A=k\theta_{1}+\frac{C_1 (r^4-4c^2)+2ck}{r^2}\theta_2+\frac{C_2}{(r^4-4c^2)^{3/2}}\theta_3+\frac{C_3}{(r^4-4c^2)^{3/2}}\theta_4.\label{spin-instanton phi1}
	\end{equation}
	The latter includes a traceless $\om_2$-HYM connection (when $C_1=C_3=0$) and a traceless $\om_3$-HYM connection (when $C_1=C_2=0$). Only the solutions with $C_2=C_3=0$ extend globally to $T^*\C\mathbb{P}^2$. 
		
	\item the general $\SU(3)$-invariant abelian $\Spin(7)$-instanton on  $N_{1,1}\times \R_{r>\sqrt{2 c}}$ with respect to $\Phi_2$ is given by
\begin{equation}
	A=k\theta_{1}+\Big(\frac{2ck}{r^2}+\frac{C_1}{r^2 (r^4-4c^2)}\Big)\theta_2+{C_2}{(r^4-4c^2)^{1/2}}\theta_3+\frac{C_3}{(r^4-4c^2)^{3/2}}\theta_4.\label{spin-instanton phi2}
\end{equation}
	The latter includes a traceless $\om_1$-HYM connection (when $C_2=C_3=0$) and a traceless $\om_3$-HYM connection (when $C_1=C_2=0$). Only the solutions with $C_1=C_3=0$ extend globally to $T^*\C\mathbb{P}^2$. 
	
	\item the general $\SU(3)$-invariant abelian $\Spin(7)$-instanton on  $N_{1,1}\times \R_{r>\sqrt{2 c}}$ with respect to $\Phi_3$ is given by
\begin{equation}
	A=k\theta_{1}+\Big(\frac{2ck}{r^2}+\frac{C_1}{r^2 (r^4-4c^2)}\Big)\theta_2+\frac{C_2}{(r^4-4c^2)^{3/2}}\theta_3+{C_3}{(r^4-4c^2)^{1/2}}\theta_4.
\end{equation}
	The latter includes a traceless $\om_1$-HYM connection (when $C_2=C_3=0$) and a traceless $\om_2$-HYM connection (when $C_1=C_3=0$). Only the solution with $C_1=C_2=0$  extends globally to $T^*\C\mathbb{P}^2$.
	\end{enumerate}
\begin{proof}
For case (1) using Propositions \ref{prop: HK structure of TCP2} and \ref{prop: curvature FA} one easily shows that $*(F_A\w \Phi_1)=-F_A$ is equivalent to the system
	\begin{gather*}
		ra_2'(r^4-4c^2)-2a_2(r^4+4c^2)+8ckr^2=0,\\
		a_3'(r^4-4c^2)+6a_3 r^3=0,\\
		a_4'(r^4-4c^2)+6a_4 r^3=0.
	\end{gather*}
	Solving  the latter yields (\ref{spin-instanton phi1}). Similarly for case (2) we have that $*(F_A\w \Phi_2)=-F_A$ is equivalent to 
	\begin{gather*}
		ra_2'(r^4-4c^2)+2a_2(3r^4-4c^2)-8ckr^2=0,\\
		a_3'(r^4-4c^2)-2a_3 r^3=0,\\
		a_4'(r^4-4c^2)+6a_4 r^3=0,
	\end{gather*}
	and the solution is given by (\ref{spin-instanton phi2}). For case (3) it suffices to swap $a_3$ and $a_4$ in case (2). 
\end{proof}
\end{Th}
Observe that the $\Spin(7)$-instantons which are globally well-defined on $T^*\C\mathbb{P}^2$ are precisely the HYM connections from Proposition \ref{prop: hym on tcp2}. Note also that the system of ODEs arising in the above theorem are all linear and moreover since the gauge group is $\U(1)$ the system is not coupled; as such solving the ODEs was rather simple. Next we shall consider the deformed instanton equations which in contrast to the above will yield systems of ODEs that are both non-linear and coupled.

\subsubsection{The deformed HYM connections and deformed $\spin(7)$-instantons}

\begin{Th}\label{theorem: dhym for om1}
	The connection $A(r)=k \theta_1+a_2\theta_2+a_3\theta_3+a_4\theta_4$ is a dHYM connection with phase $e^{i\theta}$ with respect to $\om_1$ on $T^*\C\mathbb{P}^2$ only if $a_3=a_4=0$ and $a_2(r)$ is implicitly defined by 
	\begin{equation}
	\tan(\theta)=\frac{2(a_2 r^2-2ck)(a_2^2-r^4/4+c^2-k^2)}{(a_2^2-r^4/4+c^2-k^2)^2-(a_2 r^2-2ck)^2}.\label{equ: dhym om1 neat}
	\end{equation}
	For generic values of $k,\theta$, 
	there are $2$ distinct $\om_1$-dHYM connections on $T^*\C\mathbb{P}^2$ i.e. these are the solutions $a_2(r)$ satisfying $a_2(\sqrt{2c})=k \in \mathbb{Z}$. 
\end{Th}
\begin{proof}
	First from Lemma \ref{lemma: holomorphic} we already know that we need $a_3=a_4=0$. A long but straightforward calculation using Propositions \ref{prop: HK structure of TCP2} and \ref{prop: curvature FA} then shows that
	\begin{gather*}
		F_A \w \om_1^3 =-\frac{r(r^4-4c^2)a_2'+2(3r^4-4c^2)a_2-8ckr^2}{4r^2(r^4-4c^2)} \om_1^4,\\
		F_A^4 = \frac{8a_2(a_2^2-k^2)a_2'}{r^3(r^4-4c^2)}\om_1^4,\\
		F_A^2 \w \om_1^2 = \frac{((3r^4-4c^2)a_2-4ckr^2)a_2'-2k^2r^3+6r^3a_2^2-8ckra_2}{r^3(3r^4-4c^2)}\om_1^4.
	\end{gather*}
	Using the above one finds that (\ref{equ: dhym n=4})  is given by 
	\begin{align}
		&\Big(8\tan(\theta)a_2^3-12a_2^2r^2-\big((6r^4-8c^2+8k^2)\tan(\theta)-16ck\big)a_2\label{equ: dhym om1 - ODE}\\ 
		&+r^2(r^4-4c^2+4k^2+8ck\tan(\theta))\Big)a_2' 
		-8r a_2^3-12 r^3 a_2^2 \tan(\theta)\nonumber\\
		&+(6r^5-8c^2r+8k^2r+16ckr \tan(\theta))a_2+(\tan(\theta)(r^4-4c^2+4k^2)-8ck)r^3 =0.\nonumber
	\end{align}
	It turns out that the latter is an exact ODE: indeed, it is not hard to verify that differentiating the expression
		\begin{align}
		&{2\tan(\theta)}a_2^4-4r^2a_2^3-{\big((3r^4-4c^2+4k^2)\tan(\theta)-8ck\big)}a_2^2\ \label{equ: dhym om1}+\\
		&r^2(r^4-4c^2+4k^2+8ck\tan(\theta))a_2\ \nonumber+\\
		&\tan(\theta)(\frac{r^4}{8}-c^2+k^2)r^4-2ckr^4-C=0,\nonumber
	\end{align}
	with respect to $r$ yields (\ref{equ: dhym om1 - ODE}). Thus, this shows that (\ref{equ: dhym om1}) is the general solution to (\ref{equ: dhym om1 - ODE}) and hence we have constructed all $\SU(3)$-invariant $\om_1$-dHYM connections on the open set $\R_{r>\sqrt{2 c}}\times N^{1,1}$. Observe that $a_2(r)$ is implicitly given as the solution to a quartic polynomial whose coefficients are functions of $r$.
	
	Since $C$ is an arbitrary constant of integration we have a $1$-parameter family of solutions on the open set $\R_{r>\sqrt{2 c}}\times N^{1,1}$. However, we still need to impose that $a_2(\sqrt{2 c})=k$ in order to get a smooth extension of the connection $A$ at the zero section $\C\mathbb{P}^2$. Substituting $r=\sqrt{2c}$ and $a_2(\sqrt{2 c})=k$ in (\ref{equ: dhym om1}) shows that we need $$C=-2\tan(\theta)((c^2-k^2)^2-4c^2 k^2)-8ck(c^2-k^2).$$
	With $C$ as above, one can rewrite (\ref{equ: dhym om1}) very neatly as (\ref{equ: dhym om1 neat}).
	Observe that if $\tan(\theta)=0$ then (\ref{equ: dhym om1 neat}) is a cubic polynomial in $a_2(r)$, otherwise (\ref{equ: dhym om1 neat}) is a quartic polynomial and hence generically there will be $4$ distinct solutions $a_4(r)$. Among these solutions only those satisfying $a_2(\sqrt{2c})=k$ will extend to global solutions on $T^*\C\mathbb{P}^2$. We shall determine those next.
	
	Setting $r=\sqrt{2c}$ in (\ref{equ: dhym om1 neat}) and solving for $a_2(\sqrt{2c})$ shows that $a_2(\sqrt{2c})=k$ occurs as a double root while the other two roots satisfy
	\[(a_2(\sqrt{2c})+k)^2\tan(\theta)-4c(a_2(\sqrt{2c})+k)-4c^2\tan(\theta)=0.\]
	From the latter expression we deduce easily that $a_2(\sqrt{2c})=k$ occurs as a triple root in the special case that $\tan(\theta)= \frac{2ck}{k^2-c^2}$ (this situation is illustrated in Figure \ref{figure3}); the fourth root in this case is $a_2(\sqrt{2c})=-(2c^2+k^2)/k$. 
	Hence for generic values of $k,\theta$ we get two distinct solutions satisfying $a_2(\sqrt{2c})=k$ (this is illustrated in Fig. \ref{figure1} and \ref{figure2}).
	
	To conclude the proof we still need to show that $a_2(r)$ is well-defined for all values of $r\in[\sqrt{2c},+\infty)$ since this is not a priori obvious from expression (\ref{equ: dhym om1 neat}).
	
	Let us first consider the situation when $\tan(\theta)=0$. In this case from (\ref{equ: dhym om1 neat}) we have that 
	\[(a_2 r^2 -2ck)(4a_2^2-r^4+4c^2-4k^2)=0\]
	i.e.
	\begin{equation}
		a_2 = \frac{2ck}{r^2},\ a_2 = \pm \frac{1}{2}\sqrt{r^4-4c^2+4k^2}.
	\end{equation}
	Hence, clearly there are $2$ distinct solutions with $a_2(\sqrt{2c})=k$ and they are well-defined for $r\in [\sqrt{2c},+\infty)$. 
	On the other hand, when $\cot(\theta)=0$, from (\ref{equ: dhym om1 neat}) we have that 
	\begin{align*}
	a_2=\frac{r^2}{2} &\pm \frac{1}{2}\sqrt{2r^4-4c^2-8ck+4k^2},\\ a_2=-\frac{r^2}{2} &\pm \frac{1}{2}\sqrt{2r^4-4c^2+8ck+4k^2}.
	\end{align*}
	Again it is not hard to see that we get two distinct solutions satisfying $a_2(\sqrt{2c})=k$ which are well-defined for $r\in[\sqrt{2c},\infty)$, except in case $c=-3k$ when $a_2(\sqrt{2c})=k$ occurs as a triple root as we already noted above. 
	
	For the more general case when $\cot(\theta),\tan(\theta)\neq 0$, we can view (\ref{equ: dhym om1 neat}) as a quartic polynomial in $a_2$. Setting $c=1$ for convenience, we find that the discriminant of this polynomial (up to a positive constant) is given by
	\[\Big((k^4-6k^2+1)\tan^2(\theta)-4k(k^2-1)\tan(\theta)+\frac{r^4/4+k^2-1}{\tan(\theta)}\Big)(r^4-4)^2(r^4+4k^2).\]
	When $r=\sqrt{2c}=\sqrt{2}$, we see that the above expression is zero which is consistent with the fact that there are two solutions $a_2(r)$ emanating from $(\sqrt{2c},k)$. Since the discriminant is positive for $r> 2$, we deduce that indeed $a_2(r)$ is well-defined for all $r\in [\sqrt{2c},\infty)$ and this concludes the proof.
\end{proof}

\begin{Rem}\label{rem: dhym om1} 
		There are $3$ distinct $\om_1$-dHYM connections corresponding to the case when $\theta=0$. One is the hyper-holomorphic connection of Lemma \ref{lemma: holomorphic} (as expected from Proposition \ref{prop: hyperholomorphic implies deformed}). If furthermore, we set $k=\pm c$ then the other two solutions correspond to $F_A= \mp \om_1$ (these are the elementary examples described in Section \ref{section: preliminaries}).
\end{Rem}

\begin{figure}
	\centering
	\includegraphics[height=7cm]{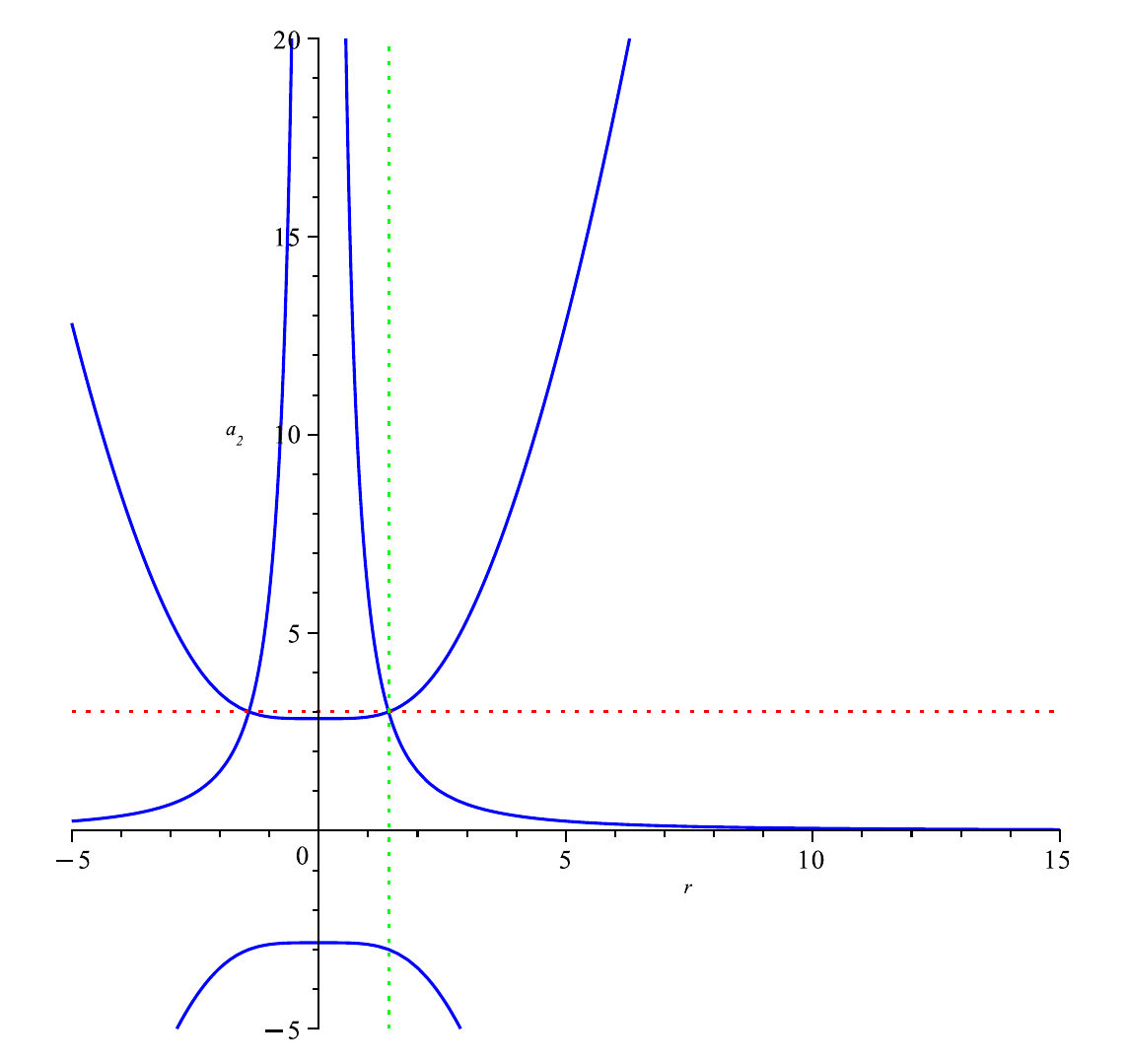}
	\caption{{Graphical solutions to (\ref{equ: dhym om1}) with $c=1$, $k=3$ and $\theta=0$ illustrating a generic case of Theorem \ref{theorem: dhym for om1}. The asymptotically zero solution corresponds to the hyper-holomorphic connection of Lemma \ref{lemma: holomorphic}}, the vertical dotted line to $r=c$ and the horizontal dotted line to $a_2=k$. There are 2 distinct solutions emanating from $(c,k)$ which are well-defined for $r\in [1,+\infty).$
	}\label{figure1}
\end{figure}
\begin{figure}
	\centering
	\includegraphics[height=7cm]{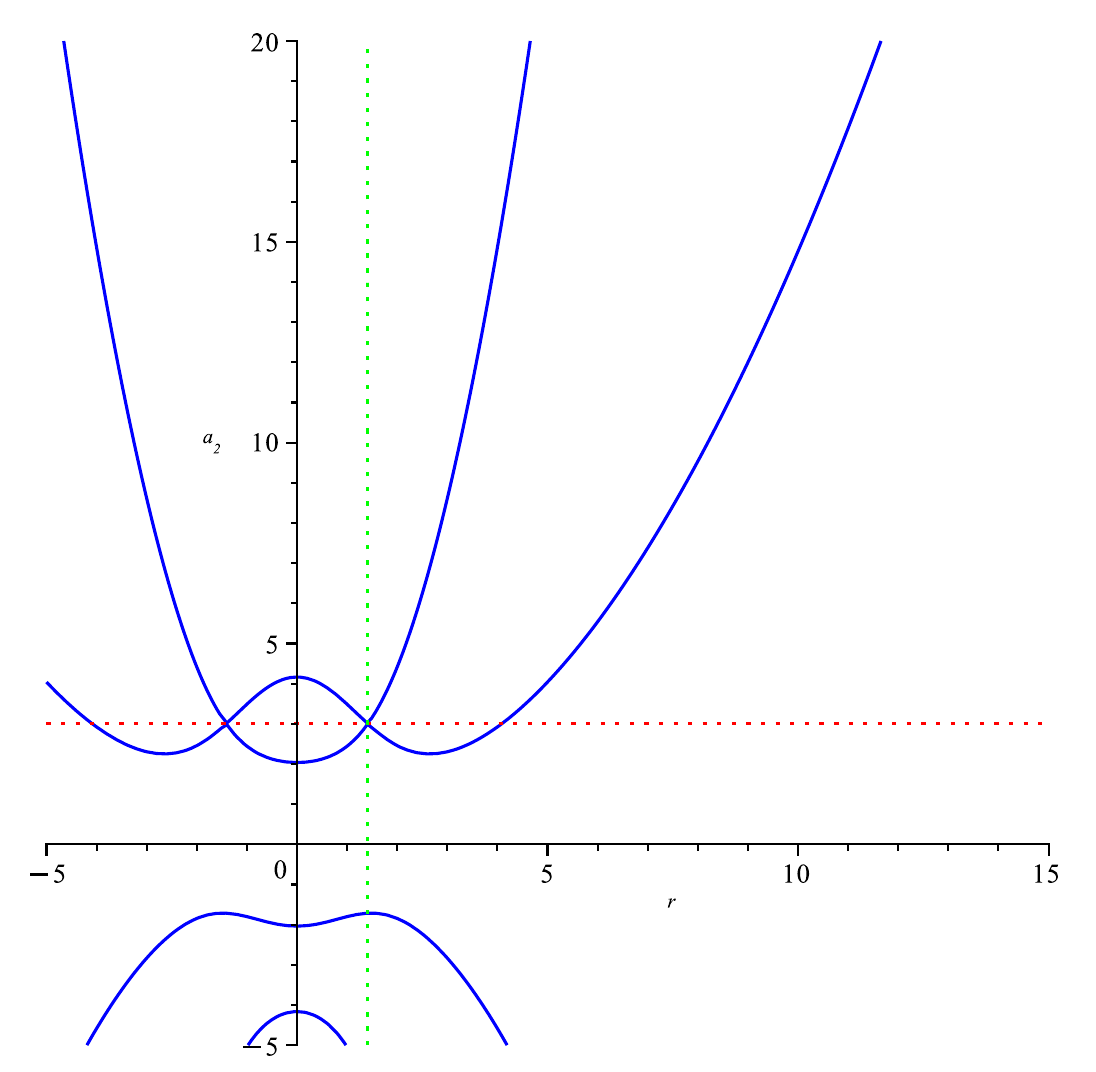}
	\caption{{Graphical solutions to (\ref{equ: dhym om1}) with  $c=1$, $k=3$ and $\theta=2$
			illustrating a generic case of Theorem \ref{theorem: dhym for om1} with non-zero phase angle $\theta$.}}\label{figure2}
\end{figure}
\begin{figure}
	\centering
	\includegraphics[height=7cm]{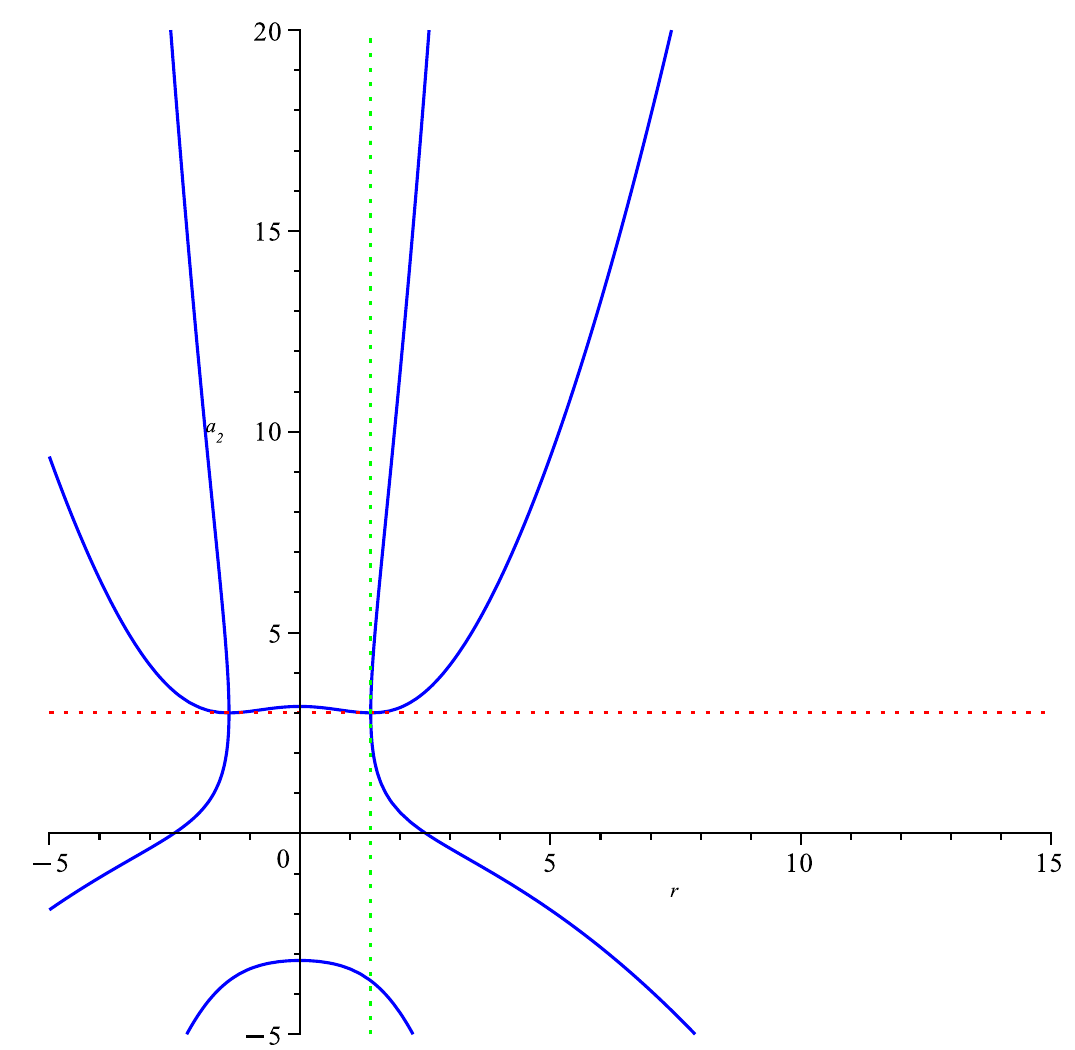}
	\caption{{Graphical solutions to (\ref{equ: dhym om1}) with $c=1$, $k=3$ and $\theta=\arctan(3/4)$ illustrating a non-generic case of Theorem \ref{theorem: dhym for om1} whereby there are 3 distinct solutions.}}\label{figure3}
\end{figure}

\begin{figure}
	\centering
	\includegraphics[height=7cm]{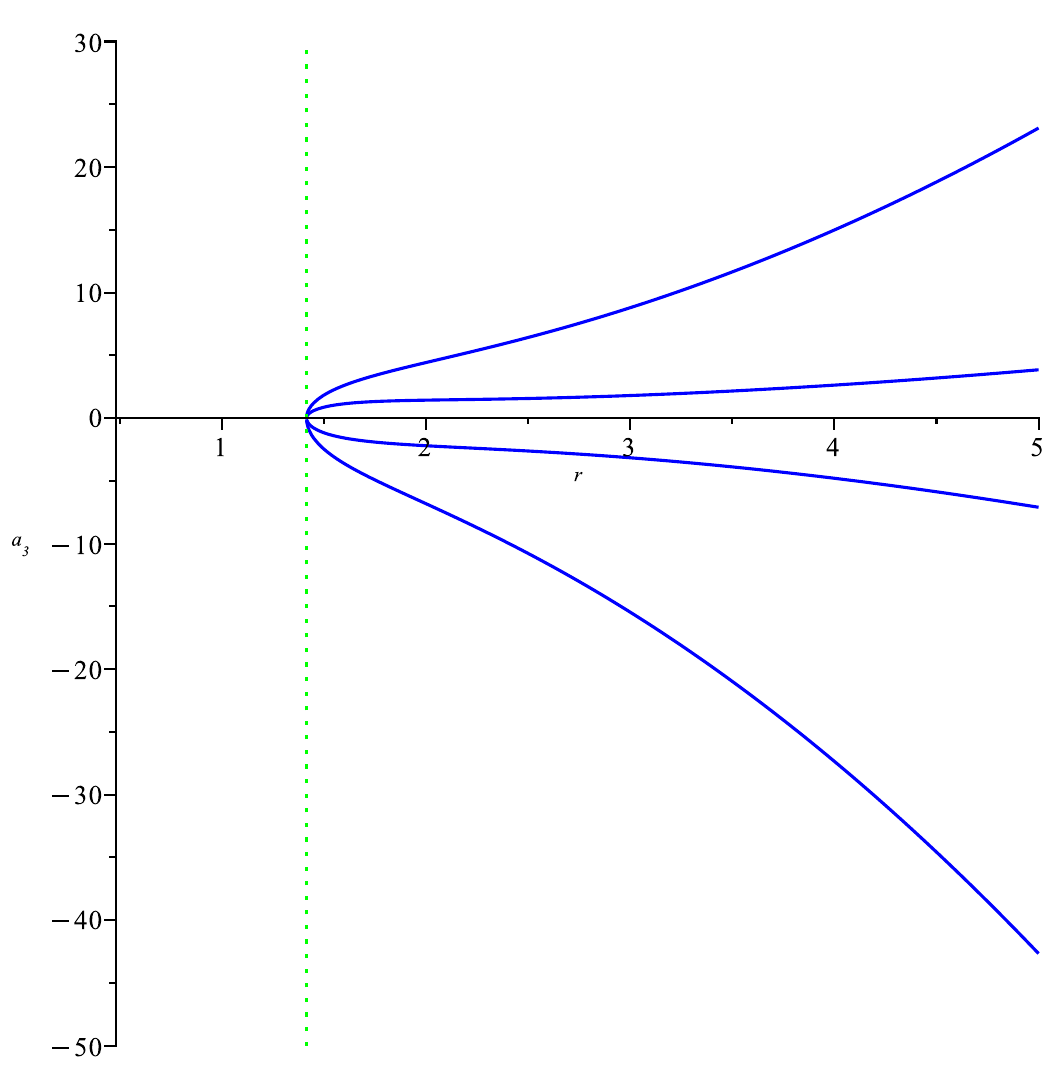}
	\caption{{Graphical solutions to (\ref{dhym om2 tcp2 solution}) with $c=k=1$ and $\theta=\arctan(2)$ illustrating a generic case of Theorem \ref{theorem: dhym for om2}.}}\label{figure4}
\end{figure}

\begin{Th}\label{theorem: dhym for om2}
The connection $A(r)=k \theta_1+a_2\theta_2+a_3\theta_3+a_4\theta_4$ is a dHYM connection with phase $e^{i\theta}$ with respect to $\om_2$ on $T^*\C\mathbb{P}^2$ if and only if $a_2= 2ck/r^2$, $a_4=0$ and $a_3(r)$ is implicitly defined by 
\begin{equation}
	\tan(\theta) =\frac{2(4a_3r^4\sqrt{r^4-4c^2})(r^8-4 r^4(a_3^2+c^2-k^2)-16c^2k^2)}{(4a_3r^4\sqrt{r^4-4c^2})^2-(r^8-4 r^4(a_3^2+c^2-k^2)-16c^2k^2)^2},\label{dhym om2 tcp2 solution}
\end{equation}
where $k \in \mathbb{Z}$. 
For generic values of $k,\theta$, there are $4$ distinct $\om_2$-dHYM connections on $T^*\C\mathbb{P}^2$ i.e. these are the solutions $a_3(r)$ satisfying $a_3(\sqrt{2c})=0$. 
\end{Th}
\begin{proof}
Since the proof is similar as in the previous theorem, we shall only highlight the main computations. From Lemma \ref{lemma: holomorphic} we already know that we need $a_2=2ck/r^2$ and $a_4=0$. A straightforward calculation then shows that 
\begin{gather*}
	F_A \w \om_2^3 =-\frac{a_3'(r^4-4c^2)+6r^3 a_3}{4r^3(r^4-4c^2)^{1/2}} \om_2^4,\\
	F_A^4 = \frac{8(r^4a_3^2-k^2r^4+4c^2k^2)(r^5a_3a_3'-8c^2k^2)}{r^{12}(r^4-4c^2)}\om_2^4,\\
	F_A^2 \w \om_2^2 = \frac{3r^5(r^4-4c^2)a_3 a_3'+6r^8 a_3^2-2k^2(r^8-16c^4)}{3r^8(r^4-4c^2)}\om_2^4.
\end{gather*}
As before computing (\ref{equ: dhym n=4}) we find that the resulting ODE is exact and that the general solution on the open set $\R_{r>\sqrt{2 c}}\times N^{1,1}$ is given by 
\begin{align}
&16\tan(\theta)a_3^4-32(r^4-4c^2)^{1/2} a_3^3-24\tan(\theta)(r^4-4c^2)(1+\frac{4k^2}{3r^4}) a_3^2+
\Big(\frac{32k^2}{r^4}+8\Big)(r^4-4c^2)^{3/2}a_3\nonumber\\
&+\tan(\theta)((r^4-4c^2)^2+\frac{8k^2}{r^8}(r^4(r^8+16c^4)-16c^2k^2(r^4-2c^2)))-C=0.\nonumber
\end{align}
Imposing that $a_3(\sqrt{2c})=0$ we find that we need 
$$C=16k^2\tan(\theta)(4c^2-k^2).$$ 
Rewriting the resulting expression with $C$ as above gives (\ref{dhym om2 tcp2 solution}). A similar argument as in the previous theorem shows that the solutions $a_3(r)$ are well-defined for all $r\in[\sqrt{2c},\infty)$ (see also Remark \ref{rem: dspin7 phi2} (3) below).
\end{proof}

\begin{Rem}\label{rem: dhym om2}
	There are $3$ distinct $\om_2$-dHYM connections corresponding to the case when $\theta=0$. One is the hyper-holomorphic connection of Lemma \ref{lemma: holomorphic} (as expected from Proposition \ref{prop: hyperholomorphic implies deformed}). If furthermore, we set $k=0$ then the other two solutions correspond to $F_A= \pm \om_2$ (these are again the elementary examples described in Section \ref{section: preliminaries}).
\end{Rem}

Note that for dHYM connections with respect to $\om_3$ it suffices to swap $a_3$ and $a_4$ in  Theorem \ref{theorem: dhym for om2}. Next we investigate the deformed $\spin(7)$ instanton condition. 

\begin{Th}\label{theorem: dspin7 for phi1}
	The connection form $A(r)=k \theta_1+a_2\theta_2+a_3\theta_3+a_4\theta_4$ is a deformed $\spin(7)$-instanton on $T^*\C\mathbb{P}^2$ with respect to $\Phi_1$ if and only if either of the following holds:
	\begin{enumerate}
		\item $a_3(r)=a_4(r)=0$ and 
			\begin{equation}
			a_2(r)=\frac{2ck}{r^2}\theta_2,\label{equ: dspin7 as hyperholo}
		\end{equation}
		\item $a_2(r)=2ck/r^2$, $a_3(r)=C_3 p(r)$ and $a_4(r)=C_4 p(r)$, where
			\begin{equation}
			p(r)=\frac{\sqrt{(r^4+4k^2)(r^4-4c^2)}}{2r^2\sqrt{C_3^2+C_4^2 }}, \label{equ: strictly dspin7}
		\end{equation}
		\item $a_3(r)=a_4(r)=0$ and 
			\begin{equation}
			a_2(r)= \Big(\frac{1}{2}C r^2 \pm \frac{1}{2}\sqrt{C^2 r^4-8 C c k+{r^4}-4c^2+4k^2}\Big),\label{equ: dspin in fact dhym}
		\end{equation}
		where if $cC\geq (\leq) k$ we take the $-$ ($+$) sign so that $a_2(\sqrt{2c})=k$.
	\end{enumerate}
\end{Th}
\begin{proof}
		First, from Proposition \ref{prop: curvature FA} we have that  
		\begin{align*}
			F_A \w F_A = & -dr\w \big(4(a_2' a_2+a_3' a_3+a_4' a_4)\theta_{234}+2a_2'(a_2+k)\theta_{256}+2a_2'a_3\theta_{2}\w\sigma_{2}
			+2a_2' a_4 \theta_{2}\w \sigma_3\\
			&+2a_2'(a_2-k)\theta_{278}+2a_3'(a_2+k)\theta_{356}+2a_3'a_3\theta_3\w \sigma_{2}+2 a_3'a_4\theta_3\w\sigma_3+2a_3'(a_2-k)\theta_{378}\\
			&+2a_4'(a_2+k)\theta_{456}+2a_4'a_3\theta_{4}\w \sigma_{2}+2a_4'a_4 \theta_4 \w \sigma_{3}+2a_4'(a_2-k)\theta_{478}\big)\\
			&+4a_4(a_2+k)\theta_{2356}+4a_3 a_4\theta_{23}\w \sigma_2+4a_4^2\theta_{23}\w \sigma_3+4a_4(a_2-k)\theta_{2378}-4a_3(a_2+k)\theta_{2456}\\
			&-4a_3^2\theta_{2457}-4a_3 a_4 \theta_{24}\w \sigma_{3}+4a_3^2\theta_{2468}-4a_3(a_2-k)\theta_{2478}+4a_2(a_2+k)\theta_{3456}\\
			&+4a_2a_3\theta_{34}\w \sigma_2+4a_2a_4\theta_{34}\w \sigma_3+4a_2(a_2-k)\theta_{3478}+2(a_2^2+a_3^2+a_4^2-k^2)\theta_{5678}.
		\end{align*}
		On the other hand from Propositions \ref{prop: spin7modulesasSp2} and \ref{prop: spaces Ei}  we know that the $\SU(3)$-invariant forms in $\Lm^4_{7}$ (defined with respect to $\Phi_1$) are spanned by $\al_2\w \om_3= - \al_3\w \om_2\in F_1^+$, $\om_1\w \om_2$ and $\om_1\w \om_3$. Hence $\pi^4_7(F_A\w F_A)=0$ is equivalent to the condition that  $F_A \w F_A$ is orthogonal to the latter triple. Using  the above expression for $F_A \w F_A$ and using Proposition \ref{prop: HK structure of TCP2} we 
		have that $g_{Ca}(F_A\w F_A,\om_1\w \om_2)=0$ is given by
		\[
		4f_0f_2f_5f_7 a_2 a_3+ 2f_0f_3f_5^2 a_3(a_2-k)+2f_0f_3f_7^2 a_3(a_2+k)-f_2f_3f_5^2a_3'(a_2-k)-f_2f_3f_7^2a_3'(a_2+k)-2f_3^2f_5f_7 a_2'a_3=0,
		\]
		where $f_i$ are as defined in Proposition \ref{prop: HK structure of TCP2}. Simplifying, it turns out that the latter is equivalent to  
		\begin{equation}
		\frac{d}{dr}((r^2a_2-2ck)(r^4-4c^2)^{1/2}a_3)=0.
		\end{equation}
		Similarly, we compute $g_{Ca}(F_A\w F_A,\al_2\w \om_3)=0$, $g_{Ca}(F_A\w F_A,\om_1\w \om_3)=0$ and find that these are given by
		\begin{gather}
			a_3'a_4=a_4' a_3,\\
			\frac{d}{dr}((r^2a_2-2ck)(r^4-4c^2)^{1/2}a_4 )=0,
		\end{gather}
		respectively. Considering all the possible distinct solutions to the above system of ODEs, we deduce that $A(r)$ is either of the form 
		\begin{equation}
			A=k\theta_1 +\frac{2ck}{r^2} \theta_2 + p(r)(C_3\theta_3+C_4\theta_4),\label{equ: tbc 2}
		\end{equation}		
		where $p(r)$ is an arbitrary function and $C_3,C_4$ are constants, or 
		\begin{equation}
			A=k\theta_1 +a_2 \theta_2 +  \frac{1}{(r^2 a_2-2ck)\sqrt{r^4-4c^2}}(C_3\theta_3+C_4\theta_4).\label{equ: tbc}
		\end{equation}
		With $A$ as given by (\ref{equ: tbc 2}), a rather long but straightforward computation using Proposition \ref{prop: HK structure of TCP2},  (\ref{equ: curvature tcp2}) and (\ref{pi27}) shows that 
		\begin{align*}
			\pi^2_7(F_A-\frac{1}{6}* F_A^3) =\ &\frac{1}{4r^8(r^4-4c^2)^{3/2}}
			\Big(r(r^4-4c^2)(12r^4(C_3^2+C_4^2)p^2-(r^4-4k^2)(r^4-4c^2))p'(r)\\
			&+8r^8(C_3^2+C_4^2)p^3-6(r^4-4c^2)(r^8+4k^2r^4/3+32c^2k^2/3)p
			\Big)(C_3 f_0f_3dr \w \theta_3\\
			&+C_4 f_0f_3dr \w \theta_4+C_4 f_2f_3\theta_{23}-C_3 f_2f_3\theta_{24}+C_3 f_5f_7\sigma_2+C_4f_5f_7 \sigma_3).
		\end{align*}
		Simplifying the above, we find that (\ref{equ: dspin7 II}) is equivalent to	
		\begin{equation*}
			\frac{d}{dr}\Big((r^4-4c^2)^{1/2}(C_3^2+C_4^2)p(r)^3-\frac{(r^4+4k^2)(r^4-4c^2)^{3/2}}{4r^4}p(r)\Big)=0
		\end{equation*}
		and hence $p(r)$ is implicitly defined as the solution to a cubic polynomial whereby the coefficients are functions of $r$.
		
		On the other hand, when $A$ as given by (\ref{equ: tbc}), similarly as above we find that (\ref{equ: dspin7 II}) is equivalent to  
		\begin{equation*}
			\frac{d}{dr}\Big(\frac{a_2^2-r^4/4+c^2-k^2}{a_2r^2-2ck}+\frac{C_3^2+C_4^2}{(a_2r^2-2ck)^3(r^4-4c^2)}\Big)=0
		\end{equation*}
		and hence $a_2(r)$ is implicitly defined by 
		\begin{align}
			&\Big(\frac{4a_2}{r^2}+\frac{8ck}{r^4}-C\Big){(r^4-4c^2)(a_2r^2-2ck)^3}-\nonumber\\
			&\frac{(r^4-4c^2)^2(r^4+4k^2)(a_2r^2-2ck)^2}{r^4}
			+4(C_3^2+C_4^2)=0.\label{local dspin7 phi1-1}
		\end{align}		
		Thus, we have found all the $\SU(3)$-invariant deformed $\spin(7)$-instantons with respect to $\Phi_1$ on $\R_{r>\sqrt{2c}}\times N^{1,1}$. To conclude it remains to apply the initial conditions to the above local solutions.
		
		Imposing $a_2(\sqrt{2c})=k$ in  (\ref{local dspin7 phi1-1}) shows that we need $C_3=C_4=0$. Solving for $a_2(r)$ we get the deformed $\spin(7)$-instanton given by (\ref{equ: dspin in fact dhym}). Similarly, imposing $p(\sqrt{2c})=0$ 
		 we get the hyper-holomorphic connection (\ref{equ: dspin7 as hyperholo}) and the deformed $\spin(7)$-instanton given by (\ref{equ: strictly dspin7}).
\end{proof}

\begin{Rem}\label{rem: dspin7 phi1}\
	\begin{enumerate}
		\item  Observe that the solution given by (\ref{equ: dspin7 as hyperholo}) corresponds to the hyper-holomorphic connection of Lemma \ref{lemma: holomorphic}. This was indeed to be expected as we showed in Proposition \ref{prop: hyperholomorphic implies deformed}.
		\item Comparing with Theorem \ref{theorem: dhym for om2}, it is easy to see that the solutions given by (\ref{equ: strictly dspin7}) correspond to the $\om_2$-dHYM connections with phase $1$ when $C_4=0$ and to the $\om_3$-dHYM connections with phase $1$ when $C_3=0$. More generally, it is not hard to verify that these deformed $\spin(7)$-instantons are in fact also dHYM connections with phase $1$ with respect to the $1$-parameter family of K\"ahler forms given by $\frac{C_3}{\sqrt{C_3^2+C_4^2}}\om_2+\frac{C_4}{\sqrt{C_3^2+C_4^2}}\om_3$; this is again a consequence of Theorem \ref{thm: result of kawai-yamamoto} (1). Furthermore, these deformed $\spin(7)$-instantons all satisfy $*F_A^4=24$, cf. Remark \ref{remark: harveylawson}.
		\item Setting $\tan(\theta)=\frac{2 C}{C^2-1}$ in (\ref{equ: dhym om1 neat}) one can verify directly that the solutions given by (\ref{equ: dhym om1 neat}) coincide with those given by (\ref{equ: dspin in fact dhym}). Thus, the deformed $\spin(7)$-instantons given by (\ref{equ: dspin in fact dhym}) in fact correspond to the $\om_1$-dHYM connections for \textit{all} phase angle $\theta$. In particular, we see that asymptotically $a_2(r) \sim (C \pm \sqrt{C^2+1})r^2/2$ i.e. $a_2$ always has quadratic growth but the coefficient depends on the phase angle. Unlike the deformed $\spin(7)$-instantons given by (\ref{equ: strictly dspin7}), these examples, as well as the hyper-holomorphic connection, do \textit{not} satisfy $*F_A^4=24$ except when $C=0$.
	\end{enumerate}
\end{Rem}

\begin{Th}\label{theorem: dspin7 for phi2}
	The connection form $A(r)=k \theta_1+a_2\theta_2+a_3\theta_3+a_4\theta_4$ is a deformed $\spin(7)$-instanton on $T^*\C\mathbb{P}^2$ with respect to $\Phi_2$ if and only if either of the following holds:
	\begin{enumerate}
		\item 	 $a_3(r)=a_4(r)=0$ and 
		\begin{equation}
			a_2(r)=\frac{2ck}{r^2}\label{hyperholomorphic dspin7 phi2}
		\end{equation}
		\item 	$a_3(r)=a_4(r)=0$ and 
		\begin{equation}
			a_2(r)=\pm\frac{1}{2}\sqrt{r^4-4c^2+4k^2}\label{dhym om1 dspin7 phi2}
		\end{equation}
		\item  $a_3(r)=0$,
		\begin{equation}
			a_2(r)=\frac{2ck}{r^2}+Ca_4(r)\cdot\frac{(r^4-4c^2)^{1/2}}{r^2}\label{gttttttt}
		\end{equation}
		and 
		\begin{equation}
			a_4(r)=\frac{4 Cck\pm r^2\sqrt{C^2(r^4-4c^2+4k^2)+r^4+4k^2}}{-2C^2(r^4-4c^2)-2r^4}(r^4-4c^2)^{1/2},\label{a44444}
		\end{equation}
		\item $a_2(r)=2ck/r^2$, $a_4(r)=0$ and 
		\begin{equation}
			a_3(r)=\frac{C r^2 \pm \sqrt{(C^2+1)r^4+4k^2}}{2r^2}(r^4-4c^2)^{1/2}.\label{a33333}
		\end{equation}
	\end{enumerate}
\end{Th}
\begin{proof}
	From Propositions \ref{prop: spin7modulesasSp2} and \ref{prop: spaces Ei} we have that the $\SU(3)$-invariant forms in $\Lm^4_{7}$ (defined with respect to $\Phi_2$) are spanned by $\al_3\w \om_1 = - \al_1\w \om_3\in F_2^+$, $\om_2\w \om_3$  and $\om_1\w \om_2$. Hence $\pi^4_7(F_A\w F_A)=0$ is equivalent to the condition that  $F_A \w F_A$ is orthogonal to the latter triple. A direct but lengthy calculation as before shows this is given by the system
	\begin{gather}
		\frac{d}{dr}(a_3 a_4 (r^4-4c^2) )=0,\label{first}\\
		(a_2 r^2 -2ck)^2 \cdot \frac{d}{dr}\Big(\frac{a_4(r^4-4c^2)^{1/2}}{a_2 r^2 -2ck}\Big)=0,\label{second}\\
		\frac{d}{dr}((a_2 r^2-2ck)(r^4-4c^2)^{1/2}a_3)=0.\label{third}
	\end{gather}	
	Imposing the initial conditions $a_3(\sqrt{2c})=a_4(\sqrt{2c})=0$ in (\ref{first}) implies that either $a_3(r)=0$ or $a_4(r)=0$.
	
	Assuming $a_3(r)=0$ automatically solves (\ref{third}) and from (\ref{second}) we find that we need $a_2(r)$ of the form (\ref{gttttttt}). Substituting the latter in (\ref{equ: dspin7 II}), a long calculation yields the exact ODE:
	\begin{align*}
		\frac{d}{dr}\Big(\frac{(C^2(r^4-4c^2)+r^4)(r^4-4c^2)^{1/2}}{r^4}a_4^3+\frac{4Cck(r^4-4c^2)}{r^4}a_4^2-\frac{(r^4+4k^2)(r^4-4c^2)^{3/2}}{4r^4}a_4\Big)=0
	\end{align*}
	Solving for $a_4(r)$ with the initial condition that $a_4(\sqrt{2c})=0$ shows that either $a_4(r)=0$, in which case we get (\ref{hyperholomorphic dspin7 phi2}), or $a_4(r)$ is given by (\ref{a44444}).
	
	On the other hand, assuming instead that $a_4(r)=0$ automatically solves (\ref{second}) and we have that  (\ref{third}) is equivalent to
	\begin{equation}
		(a_2 r^2-2ck)(r^4-4c^2)^{1/2}a_3 = C,
	\end{equation}
	where $C$ is a constant. Imposing the initial conditions $a_2(\sqrt{2c})=k$ and $a_3(\sqrt{2c})=0$ in the latter shows that $C=0$ and hence either $a_2=2ck/r^2$ or $a_3=0$. In the latter case from Lemma \ref{lemma: holomorphic} we know that $A$ is holomorphic with respect to $\om_1$. Thus, by Theorem \ref{thm: result of kawai-yamamoto} (1), the deformed $\spin(7)$-instantons with respect to $\Phi_2$ coincide with $\om_1$-dHYM connections with phase $1$ and this gives (\ref{dhym om1 dspin7 phi2}). In the former case again from Lemma \ref{lemma: holomorphic} we now see that $A$ is holomorphic with respect to $\om_2$. However, we cannot appeal to Theorem \ref{thm: result of kawai-yamamoto} as before since we are now considering deformed $\spin(7)$-instantons with respect to $\Phi_2$. With the hypothesis that $a_2=2ck/r^2$ and $a_4=0$, we find after a long computation that (\ref{equ: dspin7 II}) is equivalent to
	\begin{align*}
	\big(4a_3^2 r^4-(r^4-4c^2)(r^4+4k^2)\big)^3\frac{d}{dr}\Big(\frac{\big(4a_3^2 r^4+(r^4-4c^2)(r^4-4k^2)\big)^2+16k^2r^4(r^4-4c^2)^2}{\big(4a_3^2 r^4-(r^4-4c^2)(r^4+4k^2)\big)^2}\Big)=0. 
	\end{align*}
	Solving for $a_3(r)$ with $a_3(\sqrt{2c})=0$ gives (\ref{a33333}).
\end{proof}
Note that the deformed $\spin(7)$-instantons with respect to $\Phi_3$ can be deduced from Theorem \ref{theorem: dspin7 for phi2} by swapping $a_3$ and $a_4$. 
\begin{Rem}\label{rem: dspin7 phi2}\
	\begin{enumerate}
		\item Again as expected from Proposition \ref{prop: hyperholomorphic implies deformed} the solution given by (\ref{hyperholomorphic dspin7 phi2}) corresponds to the hyper-holomorphic connection of Lemma \ref{lemma: holomorphic}.
		\item The solutions given by (\ref{dhym om1 dspin7 phi2}) correspond to the $\om_1$-dHYM connections with phase $1$. Recall also that the $\om_3$-dHYM connections can be deduced from Theorem \ref{theorem: dhym for om2} by swapping $a_3$ and $a_4$. The solutions given by (\ref{gttttttt})-(\ref{a44444}) when $C=0$ correspond precisely to the $\om_3$-dHYM connections with phase $1$; this is again a consequence of Theorem \ref{thm: result of kawai-yamamoto} (1).  
		More generally, one can show that these deformed $\spin(7)$-instantons are in fact also dHYM connections with phase $1$ with respect to the $1$-parameter family of the K\"ahler forms given by  $\frac{C}{\sqrt{C^2+1}}\om_1+\frac{1}{\sqrt{C^2+1}}\om_3$. The solution (\ref{dhym om1 dspin7 phi2}) is obtained from (\ref{gttttttt})-(\ref{a44444}) by setting $C=1/\tilde{C}$ and taking the limit as $\tilde{C}\to 0$. It is not hard to verify that all these deformed $\spin(7)$-instantons also satisfy $*F_A^4=24$, cf. Remark \ref{remark: harveylawson}.
		\item Setting $\tan(\theta)=\frac{2 C}{C^2-1}$ in (\ref{dhym om2 tcp2 solution}) one can verify directly that the solutions given by (\ref{dhym om2 tcp2 solution}) coincide with those given by (\ref{a33333}). Note that for each fixed $C$ and $k$ (\ref{a33333}) consists of two solutions and on the other hand given $\theta$ there are in general two values of $C$ solving $\tan(\theta)=\frac{2 C}{C^2-1}$. Hence we see that this indeed yields a total of four solutions $a_3(r)$ for given $\theta$ and $k$ which is consistent with Theorem \ref{theorem: dhym for om2}.
		Thus, we have that the deformed $\spin(7)$-instantons given by (\ref{a33333}) correspond to the $\om_2$-dHYM connections for \textit{all} phase angle $\theta$. In particular, we see that asymptotically $a_3(r) \sim (C \pm \sqrt{C^2+1})r^2/2$  i.e. $a_3$ always has quadratic growth but the coefficient depends on the phase angle.
		Unlike the previous deformed $\spin(7)$-instantons, these examples do not satisfy $*F_A^4=24$ except when $C=0$.
	\end{enumerate}
\end{Rem}

Note that when $\tan(\theta)\neq 0$ equation (\ref{equ: dhym n=4}) is quartic in $F_A$ while equation (\ref{equ: dspin7 II}) is only cubic, so in general one would not expect the dHYM connections and deformed $\spin(7)$-instantons to coincide. However, as we saw in Remark \ref{rem: dspin7 phi1} (3) and \ref{rem: dspin7 phi2} (3), these instantons also satisfy $*F_A^4=24$ which might explain the above observations. It is also worth pointing out that the condition $*F_A^4 \neq 24$ is necessary in order to obtain a good description of the moduli space of deformed $\spin(7)$-instantons, see \cite[Theorem 1.2]{Kawai2021deformationdSpin7}. So it would be interesting to understand the geometrical significance of this condition. 
We also saw in this section that all the ODEs arising from the deformed instanton equations were always exact and hence we were able to find explicit solutions. However, this is not the case in general as we shall show in the next section.

\section{Deformed $\spin(7)$-instantons on $3$-Sasakian cones}\label{section: bscone}
Recall that the cone metric over a strictly nearly parallel $\G_2$ manifold has holonomy group equal to $\Spin(7)$ while the cone metric over a $3$-Sasakian $7$-manifold has holonomy group equal to $\Sp(2)$. In \cite{Lotay2020} Lotay-Oliveira constructed deformed $\G_2$-instantons on the Aloff-Wallach space $N_{1,1}$ endowed with its $3$-Sasakian and strictly nearly parallel $\G_2$-structures. They showed that the moduli space of deformed $\G_2$-instantons in each case is different: in particular, in the $3$-Sasakian case the moduli space is one dimensional while in the strictly nearly parallel case it is two dimensional cf. \cite[Corollary 4.12]{Lotay2020}. Hence their result shows that the deformed $\G_2$-instantons can be used to distinguish between these two $\G_2$-structures. In this section we shall show that a similar result also holds for deformed $\spin(7)$ instantons on the associated $\spin(7)$ cones. As before we shall again use the notation introduced in Section \ref{section: flag manifold}.

\begin{Prop}\label{prop: dspin equations for BS}
	The connection form $A(r)=k \theta_1+a_2\theta_2+a_3\theta_3+a_4\theta_4$ is a
	deformed $\spin(7)$-instanton on $N_{1,1}\times \R^+$ for the Bryant-Salamon $\spin(7)$-structure as given in Proposition \ref{prop: spin7 structure of TCP2} if and only if 
	\begin{equation}
		A = k \theta_1+ p(r)\big(C_2\theta_2+C_3\theta_3+C_4\theta_4\big),\label{connetion BS}
	\end{equation}
	where $C_i$ are constants and $p(r)$ solves
	\begin{equation}
		\big((9r^2+10c)C^2p^2+k^2r^2+100r^2(r^2+c)^{6/5}\big)p'-40r(3r^2+5c)(r^2+c)^{1/5}p+2r(C^2p^2-k^2)p=0,\label{dspin7bs}
	\end{equation}
	where $C^2:=C_2^2+C_3^2+C_4^2$. 	
\end{Prop}
\begin{proof}
Recall that in Proposition \ref{prop: spaces Ei} we showed that the space of $\SU(3)$-invariant $2$-forms are spanned by $\langle \theta_{23},\theta_{24},\theta_{34},\theta_{56},\theta_{78},\sigma_{2},\sigma_{3}, dr \w \theta_2, dr \w \theta_3, dr \w \theta_4\rangle$. In this 10 dimensional family using (\ref{pi27}) we find that $\Lm^2_7$ (defined with respect to $\Phi_{BS}$) is $3$-dimensional and in fact is spanned by 
\begin{gather*}
h_0h_2 dr \w \theta_2+h_2^2 \theta_{34}-h_5^2 \theta_{56}-h_5^2 \theta_{78},\\
h_0h_2 dr \w \theta_3-h_2^2 \theta_{24}-h_5^2 \theta_{57}+h_5^2 \theta_{68},\\
h_0h_2 dr \w \theta_4+h_2^2 \theta_{23}-h_5^2 \theta_{58}-h_5^2 \theta_{67}.
\end{gather*}
where $h_i$ are as defined in Proposition \ref{prop: spin7 structure of TCP2}. 
It follows that the space of $\SU(3)$-invariant forms in $\Lm^4_7$ is also $3$ dimensional and a rather laborious calculation shows that it is spanned by
\begin{gather*}
	h_0h_2h_5^2dr \w (\theta_{358}+\theta_{367}-\theta_{457}+\theta_{468} )-h_2^2h_5^2(\theta_{2357}-\theta_{2368}+\theta_{2458}+\theta_{2467}),\\
	h_0h_2h_5^2dr \w (\theta_{258}+\theta_{267}-\theta_{456}-\theta_{478} )-h_2^2h_5^2(\theta_{2356}+\theta_{2378}-\theta_{3458}-\theta_{3467}),\\
	h_0h_2h_5^2dr \w (\theta_{257}-\theta_{268}-\theta_{356}-\theta_{378}) +h_2^2h_5^2(\theta_{2456}+\theta_{2478}+\theta_{3457}-\theta_{3468}).
\end{gather*}
Hence $\pi^4_7(F_A \w F_A)=0$ if and only if $F_A \w F_A$ is orthogonal to latter triple. A direct computation shows that $\ref{equ: dspin7 I}$ is equivalent to
\begin{gather}
	a_2' a_3 = a_2 a_3',\\
	a_2' a_4 = a_2 a_4',\\
	a_3' a_4 = a_3 a_4'.
\end{gather} 
Thus, it follows that we need $A$ of the form (\ref{connetion BS}).
Using Proposition \ref{prop: curvature FA}, a long calculation shows that 
\begin{align*}
	F_A - \frac{1}{6}*F_A^3 =\ &\Big(\frac{50r(r^2+c)^{6/5} p'+C^2 p^3-k^2p}{200 (r^2+c)^{4/5}r^2}\Big) h_0 h_2 (\theta_1 \w (C_2\theta_2+C_3\theta_3+C_4\theta_4)) -\\
	&\Big(\frac{200(r^2+c)^{6/5} p^2+r(C^2p^2-k^2)p'}{400 (r^2+c)^{4/5}r^2}\Big) h_2^2  (C_2\theta_{34}+C_3\theta_{42}+C_4\theta_{23})+\\
	&\Big(\frac{C^2(r^2+c)^{4/5}(k-C_2p)pp'-8r^3(C_2p+k)}{80 (r^2+c)^{3/5}r^3}\Big) h_5^2  \theta_{56}+\\
	&\Big(\frac{C^2(r^2+c)^{4/5}(-k-C_2p)pp'-8r^3(C_2p-k)}{80 (r^2+c)^{3/5}r^3}\Big) h_5^2  \theta_{78}-\\	
	&\Big(\frac{(C^2(r^2+c)^{4/5}pp'+8r^3)p}{80 (r^2+c)^{3/5}r^3}\Big) h_5^2  (C_3\sigma_2+C_4\sigma_3).
\end{align*}
Using (\ref{pi27}) we can then compute (\ref{equ: dspin7 II}) and another lengthy computation yields (\ref{dspin7bs}).
\end{proof}
In contrast to the ODEs we obtained in Section \ref{section: main}, equation (\ref{dspin7bs}) is not exact and hence we have been unable to find the general solution to (\ref{dspin7bs}). However in the special case when $k=c=0$, i.e. when $g_{BS}$ corresponds to the conical metric and $L$ is the trivial line bundle, we have the following result:
\begin{Th}\label{thm: dspin7 cone}
Consider $N_{1,1}\times \R^+$ endowed with the Bryant-Salamon conical $\spin(7)$-structure of Proposition \ref{prop: spin7 structure of TCP2}. Then the $\SU(3)$-invariant connection $A(r)=a_2\theta_2+a_3\theta_3+a_4\theta_4$ is a deformed $\spin(7)$-instanton 
if and only if 
\begin{equation}
	A =  \frac{10 r^{6/5}}{3\sqrt{W(C_0  r^{128/45})}}\Big(\frac{C_2\theta_2+C_3\theta_3+C_4\theta_4}{\sqrt{C_2^2+C_3^2+C_4^2}}\Big),\label{dspin7conical}
\end{equation}
where $W$ denotes the Lambert W-function, implicitly defined by $x=W(x)\exp(W(x))$, and $C_0$ is a positive constant. On the other hand we have that $A$ is a $\spin(7)$-instanton if and only if
\begin{equation}
	A =  r^{6/5}\big(C_2\theta_2+C_3\theta_3+C_4\theta_4\big).\label{spin7conical}
\end{equation}
\end{Th}
\begin{proof}
Setting $c=k=0$ in (\ref{dspin7bs}) we get
\begin{equation*}
	\big(9C^2rp^2+100r^{17/5}\big)p'-120r^{12/5}p+2C^2p^3=0.
\end{equation*}
Solving the latter yields (\ref{dspin7conical}). On the other hand, a straightforward computation shows that $*(F_A\w \Phi_{BS})=-F_A$ is equivalent to the system
\[r(r^2+c)a_i'-2(\frac{3r^2}{5}+c)a_i=0.\]
The solution with $c=0$ is given by (\ref{spin7conical}).
\end{proof}

Note that since asymptotically $W(x)\sim \log(x) +o(1)$, it follows that the deformed $\Spin(7)$-instantons given by (\ref{dspin7conical}) have slower growth rate than the $\Spin(7)$-instantons given by (\ref{spin7conical}). We should point out that Theorem \ref{thm: dspin7 cone} does not only apply to the Aloff-Wallach space $N_{1,1}$ but to any $3$-Sasakian $7$-manifold with its squashed strictly nearly parallel $\G_2$-structure as given by \cite[Proposition 2.4]{GalickiSalamon1996}; in this case the $1$-forms  $\theta_{2},\theta_{3},\theta_{4}$ in Theorem \ref{thm: dspin7 cone} are then replaced by the dual $1$-forms to the Reeb Killing vector fields. Moreover, it is worth noting that Proposition \ref{prop: dspin equations for BS} with $k=0$ also corresponds to the deformed $\Spin(7)$ condition on the trivial line bundle over the spinor bundle of $S^4$ endowed with the \textit{complete} Bryant-Salamon $\Spin(7)$ metric \cite{Bryant1989}: this is simply because the structure equations are the same. In this case the initial condition we need for $A$ to extend smoothly to the zero section $S^4$ is $p(0)=0$. Thus, equation (\ref{dspin7bs}) corresponds to a singular initial value problem and as such one cannot appeal standard theory to obtain existence and uniqueness of a solution. With $k=0$ and $C=c=1$, seeking for a power series solution to (\ref{dspin7bs}) around $r=0$ gives
\[
p(r)=a-\frac{a^2-100}{10a} r^2+\frac{a^4-2000}{20a^3}r^4+\cdots
\]
where $a\in \R$. In particular, we see that such a solution cannot vanish at $r=0$. We saw a similar phenomenon when investigating the deformed $\G_2$-instanton equations on the spinor bundle of $S^3$  in \cite[Section 4.2]{FowdardG2}, but nonetheless there can still exist smooth solutions as demonstrated therein. We hope to analyse equation (\ref{dspin7bs}) in more detail in future work to find global solutions. 
Note that we also obtained singular initial value problems in the last section but this was not a difficulty since we were already able to solve the equations explicitly.

We now want to compare the above instantons on the strictly nearly parallel $\G_2$ cone with those on the $3$-Sasakian cone. From the results of Section \ref{section: main} we can easily deduce the following:
\begin{Th}\label{thm: dspin7 cone HK}
Consider $N_{1,1}\times \R^+$ endowed with the Calabi conical hyperK\"ahler structure of Proposition \ref{prop: HK structure of TCP2}.
Then the $\SU(3)$-invariant connection $A(r)=a_2\theta_2+a_3\theta_3+a_4\theta_4$ is a
deformed $\spin(7)$-instanton with respect to $\Phi_1$  if and only if 
either $A=C_2r^2\theta_2$ or
\begin{equation}
	A=p(r)\Big(\frac{C_3\theta_3+C_4\theta_4}{2\sqrt{C_3^2+C_4^2}}\Big),\label{oooooo}
\end{equation}
where $p(r)$ is implicitly defined by
\begin{equation}
	p^3-{r^4}p-\frac{C_0}{r^2}=0.\label{oooooop}
\end{equation}
On the other hand we have that $A$ is a $\spin(7)$-instanton with respect to $\Phi_1$ if and only if
\begin{equation}
	A =  C_2r^2\theta_2+\frac{1}{r^6}\Big(C_3\theta_3+C_4\theta_4\big).
\end{equation}
	The instantons with respect to $\Phi_2$ and $\Phi_3$ are obtained by permuting $\theta_2,\theta_3$ and $\theta_4$ in the above expressions for $A$.
\end{Th}
Observe that when $A=C_2r^2\theta$ in Theorem \ref{thm: dspin7 cone HK}, we have that $F_A$ is a constant multiple of $\om_1$ and this corresponds to both a $\spin(7)$-instanton and deformed $\spin(7)$-instanton with respect to $\Phi_1$. On the other hand when $C_0=0$ and $p(r)=\pm r^2$, from (\ref{oooooo}) we have that  $F_A=\mp\frac{C_3\om_2+C_4\om_3}{\sqrt{C_3^2+C_4^2}}$ and this corresponds to a deformed $\spin(7)$-instanton  with respect to $\Phi_1$, but not a $\spin(7)$-instanton. These examples coincide with the elementary examples described in Section \ref{section: preliminaries}. Hence we only get non-elementary examples when $C_0\neq 0$, in which case from (\ref{oooooop}) we see that asymptotically either $p(r) \sim \pm r^2 $ or $p(r) \sim 0$. 
Furthermore, we should also point out that for each $C_0 \neq  0$ exactly one of these three solutions is well-defined on $\R^+$, while the other two are only defined on a subset of $\R^+$. 

Comparing the deformed $\spin(7)$-instantons of Theorem \ref{thm: dspin7 cone} and \ref{thm: dspin7 cone HK}, we see that the moduli spaces in each case have different dimensions and furthermore, they have different number of connected components. 

\begin{Rem}\label{rem: analogies}
In \cite[Corollary 4.5]{Lotay2020} Lotay-Oliveira showed that there exist a circle of non-trivial deformed $\G_2$-instantons on $N_{1,1}$ with respect to (one of) its $3$-Sasakian $\G_2$-structure and a $2$-sphere of non-trivial deformed $\G_2$-instantons with respect to its strictly nearly parallel $\G_2$-structure. In particular, this shows that deformed $\G_2$-instantons can discriminate between these two $\G_2$-structures cf. \cite[Remark 4.6]{Lotay2020}.  
If we fix the constant $C_0$ in Theorem \ref{thm: dspin7 cone} and \ref{thm: dspin7 cone HK} above and we restrict the deformed $\spin(7)$-instantons $A(r)$ to a hypersurface $r=r_0$, then we obtain a circle and $2$-sphere family of connections respectively; these coincide precisely with the aforementioned deformed $\G_2$-instantons of Lotay-Oliveira up to an overall constant factor. 
\end{Rem}

\noindent \textbf{Declarations.}
Data sharing is not applicable to this article as no datasets were generated or analysed during the current study. The author has no conflicts of interest to declare that are relevant to the content of this article.

\bibliography{biblioG}
\bibliographystyle{plain}

\end{document}